\documentclass[final]{siamltex}
\usepackage{amsfonts}
\usepackage{graphics}

\usepackage{xcolor}\usepackage{amssymb}
\usepackage{psfrag}
\usepackage{amsmath}
\usepackage{graphicx}
\usepackage{bm}
\usepackage{caption2}
\usepackage{epstopdf}

\usepackage[colorlinks,linkcolor=black,anchorcolor=black,citecolor=black]{hyperref}
\usepackage{cleveref}

\newcommand{\be}{\begin{equation}}
\newcommand{\ee}{\end{equation}}
\usepackage{setspace}
\usepackage{amsmath,amssymb,bbm}
\usepackage{graphicx}
\usepackage[utf8]{inputenc}
\usepackage{url}
\usepackage{color}

\newtheorem{lem}{Lemma}
\newtheorem{thm}{Theorem}
\newtheorem{assum}{Assumption}

\newtheorem{rem}{Remark}

\newcommand{\bfx}{\bm{x}}
\newcommand{\bfe}{\bm{e}}
\newcommand{\bfX}{\bm{X}}
\newcommand{\bfW}{\bm{W}}
\newcommand{\bfY}{\bm{Y}}
\newcommand{\bfZ}{\bm{Z}}
\newcommand{\bfu}{{\bm{v}}}
\newcommand{\ii}{\mathbf{i}}

\begin{document}

\title{
Numerical Methods for a Class of Nonlocal Diffusion Problems 
with the Use of Backward SDEs
\thanks{This material is based upon work supported in part by the U.S.~Air Force of Scientific Research under grant numbers FA9550-11-1-0149 and 1854-V521-12; by the U.S.~Department of Energy, Office of Science, Office of Advanced Scientific Computing Research, Applied Mathematics program under contract numbers ERKJ259, and ERKJE45; by the National Natural Science Foundations of China under grant numbers 91130003 and 11171189; by Natural Science Foundation of Shandong Province under grant number ZR2011AZ002;
and by the Laboratory Directed Research and Development program at the Oak Ridge National Laboratory, which is operated by UT-Battelle, LLC, for the U.S.~Department of Energy under Contract DE-AC05-00OR22725.}
}


\author{G.~Zhang\thanks{Department of Computational and Applied Mathematics, Oak Ridge National Laboratory, 
Oak Ridge, TN 37831({\tt zhangg@ornl.gov, webstercg@ornl.gov}).}
        \and W.~Zhao\thanks{School of Mathematics, Shandong University, Jinan 250100, China ({\tt wdzhao@sdu.edu.cn}).}
         \and C.~G.~Webster\footnotemark[2]
         \and M.~Gunzburger\thanks{Department of Scientific Computing, Florida State University ({\tt gunzburg@fsu.edu}).}
}

\maketitle

\pagestyle{myheadings}
\markboth{G.~Zhang, W.~Zhao, C.~G.~Webster and M.~Gunzburger}
         {Numerical Solution of Nonlocal Diffusion Equations}

\begin{abstract}
We propose a novel numerical approach for nonlocal diffusion equations \cite{Du:2012hp} with integrable kernels, 
based on the relationship between the backward Kolmogorov  equation and backward stochastic differential equations (BSDEs) driven by L\`{e}vy processes with jumps.  The nonlocal diffusion problem under consideration is converted to a BSDE,
for which numerical schemes are developed and applied directly. As a stochastic approach, the proposed method does not require the solution of linear systems, which allows for embarrassingly parallel implementations and also enables adaptive approximation techniques to be incorporated in a straightforward fashion. Moreover, our method is more accurate than classic stochastic approaches due to the use of high-order temporal and spatial discretization schemes. In addition, our approach can handle a broad class of problems with general nonlinear forcing terms as long as they are globally Lipchitz continuous. Rigorous error analysis of the new method is provided as several numerical examples that illustrate the effectiveness and efficiency of the proposed approach.
%
%
\end{abstract}





%
%
\section{Introduction}\label{sec:intro}

A diffusion process is deemed nonlocal when the associated underlying L\`{e}vy process does not only consist of Brownian motions. 
The features of nonlocal diffusion has been verified experimentally to be present in a wide variety of applications, including, e.g., contaminant flow in groundwater, plasma physics, and the dynamics of financial markets. 
%
A comprehensive survey of nonlocal diffusion problems is given in \cite{Metzler:2000du}. In this work, we consider a partial-integral diffusion equation (PIDE) representation \cite{Chen:2011fd,Du:2013gf,Du:2012hp} of linear nonlocal diffusion problems. Since it is typically difficult to obtain analytical solutions of such problems, numerical solutions are highly desired in practical applications. There are mainly two types of numerical methods for nonlocal diffusion equations under consideration. 
The first is the family of deterministic approaches, e.g.,~finite-element-type methods \cite{Chen:2011fd,Du:2013gf} and collocation methods, 
whereas the second can be classified as stochastic approaches, e.g., continuous-time random walk (CTRW) methods 
\cite{2011PhRvE..83a2105B,Anonymous:HOINcelb,Metzler:2000du}. 

The deterministic approaches are extensions of existing methods for local partial differential equations (PDEs) that incorporate schemes to discretize the underlying integral operators. The recently developed nonlocal vector calculus \cite{Du:2013jn} provides helpful tools that allow one to analyze nonlocal diffusion problems in a similar manner to analyzing local PDEs; see \cite{Du:2012hp,Du:2013gf} for details.
However, in the context of implicit time-stepping schemes,
the nonlocal operator may result in severe computational difficulties
coming from the dramatic deterioration of the 
sparsity of the stiffness matrices required by the underlying linear systems. 
On the other hand, stochastic approaches, e.g., CTRW methods, are based on the relation between nonlocal diffusion and a class of L\`{e}vy jump processes. Although stochastic methods do not require the solution of linear systems, and the simulations of all paths of the random walk can be easily parallelized, they also have some drawbacks. 
Typically, most stochastic methods are sampling-based Monte Carlo approaches, 
suffering from slow convergence, thus requiring a very large number of samples to achieve small errors.  
Even worse, in the case of a general nonlinear forcing term, 
the nonlocal diffusion equation is no longer the master equation of the  underlying jump processes. 

In this paper, we propose a novel stochastic numerical scheme for the linear nonlocal diffusion problems studied in \cite{2011PhRvE..83a2105B,Chen:2011fd,Anonymous:HOINcelb,Du:2013jn,Du:2013gf,Du:2012hp} based on 
the relationship between the PIDEs and a certain class of backward stochastic differential equations (BSDEs) with jumps.
The existence and uniqueness of solutions for nonlinear BSDEs with and without jumps have been proved in \cite{Pardoux:1990ju} and \cite{Anonymous:fk}, respectively. 
Since then, BSDEs have become important tools in probability theory, stochastic optimal control and mathematical finance \cite{Delong:2013uk}. 
Unlike BSDEs driven by Brownian motions, there are very few numerical schemes proposed 
for BSDEs with jumps, and most of the schemes are focussed on time discretizations only.
For example, a numerical scheme  based on Picard's method was proposed in \cite{Lejay:2007tu}, and a forward Euler scheme was proposed in ~\cite{Bouchard:2008cp,Bouchard:2009il} where the convergence rate was proved to be $(\Delta t)^{\frac{1}{2}}$. 

Our focus will be on the BSDEs corresponding to nonlocal diffusion equations with {\em integrable} jump kernels in unbounded domains. Since the PIDEs of interest do not possess local convection and diffusion terms, the corresponding BSDEs can be simplified, such that, the underlying L\`{e}vy processes are merely {\em compound Poisson processes}. 
Also, motivated by various engineering applications discussed in \cite{Du:2012hp}, our goal is to develop and analyze 
computationally efficient numerical schemes that can achieve high-order accuracy in both temporal and spatial discretization, rather than focus on the scalability of our schemes in solving extremely high-dimensional problems (e.g., in option pricing).
From this perspective, although the BSDEs under consideration is a special case of the equations studied in \cite{Bouchard:2008cp}, high-order fully-discrete schemes and the corresponding error analysis are still missing in the literature. In fact, based on our previous work, 
it is particularly challenging to recover the classic finite element convergence rates in the spatial discretization, even for the simplified BSDEs. 
Thus, the main contributions of this paper are summarized as follows:
\begin{itemize}
\item Development of stable high-order matrix-free numerical schemes for the BSDEs and the corresponding nonlocal diffusion problems; 
\item Analysis of the approximation error of the proposed semi-discrete and fully-discrete schemes, and;
\item Numerical demonstration of the capabilities of handling multiple types of integrable kernels, e.g., symmetric, non-symmetric, or singular kernels.
\end{itemize}

Specifically, 
the BSDEs will be discretized, in the temporal domain, using a $\theta$-scheme extended from our previous works \cite{Zhao:2012ku,Zhao:2006jx} for BSDEs without jumps.
%
%
In particular, the cases where $\theta = 0$, $\theta= \frac{1}{2}$ and $\theta = 1$ correspond to forward Euler, Crank-Nicolson and backward Euler schemes, respectively. 
As discussed in \cite{Zhao:2012ku}, a quadrature rule adapted to the underlying stochastic processes is critical to approximate all the conditional mathematical expectations in our numerical scheme. Thus, we propose a general formulation of the quadrature rule for estimating the conditional expectations with respect to the compound Poisson processes, and a specific form of a certain PIDE can be determined based on regularities of the kernel and forcing term. In \S\ref{sec:ex}, Gauss-Legendre, Gauss-Jacobi and Newton-Cotes rules are substituted into the proposed quadrature rule to approximate different kernels and forcing terms. 
In addition, since the total number of quadrature points grows exponentially with the number of time steps, we construct a piecewise Lagrange interpolation scheme based on a pre-determined spatial mesh, in order to evaluate the integrand of the expectations at all quadrature points. 

%
%
Both theoretical analysis and numerical experiments show that our approach is advantageous for linear nonlocal diffusion equations with integrable kernels on unbounded domains. 
Compared to deterministic approaches, e.g., finite elements and collocation approaches, {\em the proposed methods do {not} require the solution of {dense} linear systems}.
Instead, the PIDE can be solved independently at different spatial grid points on each time level, making it straightforward to incorporate parallel implementation and adaptive spatial approximation. 
Compared to stochastic approaches, our scheme is more accurate than the CTRW method due to the use of the $\theta$-scheme for time discretization, the high-order quadrature rule, and piecewise Lagrange polynomial interpolation for spatial discretization. In addition, our method can handle a broad class of problems with general nonlinear forcing terms as long as they are globally Lipchitz continuous.

The outline of the paper is as follows. In \S\ref{sec:nlocBSDE}, we introduce the mathematical description of the linear nonlocal diffusion equations and the corresponding class of BSDEs under consideration. In \S\ref{sec:scheme}, we propose our numerical schemes for the BSDEs of interest, where the semi-discrete and the fully-discrete approximations are presented in \S\ref{sec:semi} and \S \ref{sec:fully}, respectively. Error estimates for the proposed scheme are proved in \S\ref{sec:err}. Extensive numerical examples and comparisons to existing techniques are given in \S \ref{sec:ex}, which 
are shown to be consistent with the theoretical results and reveal the effectiveness of our approach.  Finally, several concluding remarks are given in \S\ref{sec:con}.

%
%
\section{Nonlocal diffusion models and BSDEs with jumps}\label{sec:nlocBSDE}
This section is dedicated to describing the nonlocal diffusion problem that is the focus of this paper.  In particular, 
in \S\ref{sec:nloc1} we  introduce the definitions of a general nonlocal operator and diffusion equation. The relationship between 
the backward Kolmogorov equation, which is a generalization of the nonlocal diffusion equation of interest,
and the BSDE with jumps, is described in \S\ref{sec:BSDE}. Based on this equivalence, 
a simplified BSDE corresponding to the nonlocal diffusion problem of \S\ref{sec:nloc1} is also given at the end of \S\ref{sec:BSDE}.

\subsection{Nonlocal diffusion equations}\label{sec:nloc1}
Let us recall a nonlocal operator introduced in \cite{Du:2012hp}. 
For a function $u(t,\bm{x}): [0,T] \times \mathbb{R}^d \rightarrow \mathbb{R}$, with 
 $d = 1,2,3$ and $T >0$, we define the action of the linear operator $\mathcal{L}$ on 
$u(t, \bm{x})$ as 
\begin{equation}\label{eq:L}
  \mathcal{L}u = \int_{\mathbb{R}^d} \big(u(t, \bfx+\bfe)-u(t,\bfx)\big) \, \gamma(\bfe) 
  \, d \bfe, \quad \forall (t,\bfx) \in [0,T] \times \mathbb{R}^d,
\end{equation}
where the properties of $\mathcal{L}$ 
depend crucially on the kernel function $\gamma(\bfe): \mathbb{R}^d \rightarrow \mathbb{R}$.
In this work, we focus on a particular class of kernel functions which are nonnegative and integrable, i.e.,
\begin{equation}\label{gamma}
\gamma(\bfe) \ge 0 \;\; \forall \bfe \in \mathbb{R}^d \; \text{ and }\; \int_{\mathbb{R}^d} \gamma(\bfe)\; d\bfe < \infty.
\end{equation} 
Note that $\gamma(\bfe)$ may be {\em symmetric}, i.e.,
$\gamma(\bfe) = \gamma(-\bfe)$ for any $\bfe \in \mathbb{R}^d$, or {\em non-symmetric}, i.e., there exists $\bfe \in \mathbb{R}^d$ such that $\gamma(\bfe) \neq \gamma(-\bfe)$.
The operator $\mathcal{L}$ exhibits nonlocal behavior because, for a fixed $ t\in [0, T]$, the value of $\mathcal{L}u$ at a point $\bfx = (x^1, \ldots, x^d) \in \mathbb{R}^d$ requires information about $u$
at points $\bfx+\bfe \neq \bfx$; 
this is contrasted with local operators, e.g., the value of $\Delta u$ at a point $\bfx$ requires information about $u$
only in an infinitesimal neighborhood of $\bfx$. Our interest in the operator $\mathcal{L}$ is due to its participation in the  {\em nonlocal diffusion equation}
\begin{equation}\label{nloc_diff}
\left\{
\begin{aligned}
& \frac{\partial u}{\partial t}(t,\bfx) - \mathcal{L} u(t,\bfx)= g(t,\bfx,u) & \text{ for } (t,\bfx) \in (0,T] \times \mathbb{R}^d\\
& u(0,\bfx) = u_0(\bfx), \text{ for } \bfx \in \mathbb{R}^d,\\
\end{aligned}\right.
\end{equation}
where $u_0(\bfx)$ is the initial condition and the forcing term $g(t,\bfx,u)$ may be a {\em nonlinear} function of $t, \bfx$ and $u$. 
We remark that  \eqref{nloc_diff} is a nonlocal Cauchy problem defined on the {\em unbounded} spatial domain $\mathbb{R}^d$. 
Also, in the context of kernel functions $\gamma(\bfe)$ that are compactly supported in $\mathbb{R}^d$, the initial-boundary value nonlocal problem with {\em volume constraints}, i.e., constraints applied on a regions of nonzero measure in $\mathbb{R}^d$, has been well studied \cite{Chen:2011fd,Du:2012hp,Du:2013gf}. However, the well-posedness of the corresponding BSDEs with volume constraints is still an open challenge.  As such, we do not impose volume constraints to \eqref{nloc_diff} in this effort. 
%
%

\subsection{BSDEs and backward Kolmogorov equations}\label{sec:BSDE}
Now we discuss the probabilistic representation of the solution of the nonlocal diffusion equation in \eqref{nloc_diff}.
The nonlinear Feynman-Kac formula studied in \cite{Pardoux:1992jo} shows that BSDEs driven by Brownian motion 
provides a probabilistic representation of the solutions of a class of second-order quasi-linear parabolic PDEs. This result was
extended in \cite{Anonymous:fk} to the case of BSDEs driven by general L\'{e}vy processes with jumps. 
Due to the inclusion of L\`{e}vy jumps, the counterpart of a BSDE is a partial integral differential equation (PIDE), i.e., the backward Kolmogorov equation \cite{Hanson:2007ty}. It turns out that such a PIDE is a generalization of the nonlocal diffusion equation in \eqref{nloc_diff}. 
Here we first introduce a general form of a BSDE with jumps and the backward Kolmogorov equation, then 
give a simplified form of the BSDE corresponding to the nonlocal diffusion equation \eqref{nloc_diff} under consideration.

Let $(\Omega, \mathcal{F}, \mathbb{P})$ be a probability space with a filtration $\{\mathcal{F}_t\}_{0\le t\le T}$ 
for a finite terminal time $T>0$. We assume the filtration $\{\mathcal{F}_t\}_{0\le t\le T}$ satisfies the usual hypotheses of completeness, i.e., $\mathcal{F}_0$ 
contains all sets of $\mathbb{P}$-measure zero and maintains right continuity, i.e., $\mathcal{F}_t = \mathcal{F}_{t+}$. Moreover, the filtration is assumed to be generated by two mutually independent processes, i.e., one $m$-dimensional Brownian motion $\bm{W}_t = (W_t^1, \ldots, W_t^m)^{\top}$ and one 
$d$-dimensional Poisson random measure $\mu(A,t)$ on $E \times [0,T]$, where 
$E = \mathbb{R}^d \backslash\{0\}$ is equipped with its Borel field $\mathcal{E}$. 
The compensator of $\mu$ and the resulting compensated Poisson random measure are
 denoted by $\nu(d\bfe,dt) = \lambda(d\bfe)dt$ and $\tilde{\mu} (d\bfe, dt)= \mu(d\bfe, dt) - \lambda(d\bfe)dt$ respectively, such that $\{\tilde{\mu}(A \times [0,t]) = (\mu - \nu)(A \times [0,t])\}_{0\le t \le T}$ is a martingale for all $A \in \mathcal{E}$.
 We also assume that $\lambda(d\bfe)$ is a $\sigma$-finite measure on $(E, \mathcal{E})$ satisfying 
\begin{equation}\label{lambda}
\int_{E} (1 \land |\bfe |^2) \lambda (d \bfe) < + \infty.
\end{equation}
Based on the stochastic basis $(\Omega, \mathcal{F}, \{\mathcal{F}\}_{0\le t \le T},\mathbb{P})$, we introduce the 
backward stochastic differential equation with jumps
\begin{equation}\label{eq:BSDE}\left\{
\begin{aligned}
\bfX_t&= \bfX_0 + \int_0^t \bm{b}(s,\bfX_s) ds +  \int_0^t\bm{\sigma}(s,\bfX_s) d\bfW_s+ \int_0^t\int_{E} \bm{\beta}(s,\bm{X}_{s-},\bfe)\tilde{\mu}(d\bfe,ds)\\
\bfY_t&= \bm{\xi} + \int_t^T \bm{f}(s,\bfX_s, \bfY_s, \bfZ_s, \bm{\Gamma}_s) ds - \int_t^T \bfZ_s d\bfW_s -  \int_t^T \int_{E} \bm{U}_s(\bfe)\tilde{\mu}(d\bfe,ds),
\end{aligned}\right.
\end{equation}
where the solution is the quadruplet $(\bfX_t, \bfY_t, \bfZ_t, \bm{U}_t)$ with $\bfX_t \in \mathbb{R}^d$, $\bfY_t \in \mathbb{R}^q$, 
$\bfZ_t \in \mathbb{R}^{q\times m}$ and $\bm{U}_t \in \mathbb{R}^q$. 
The map
$\bm{b}: [0,T] \times \mathbb{R}^{d} \rightarrow \mathbb{R}^d$ is referred to as the drift coefficient, 
$\bm{\sigma}: [0,T] \times \mathbb{R}^d \rightarrow \mathbb{R}^{d} \times \mathbb{R}^m$ is the local diffusion coefficient,
$\bm{\beta}: [0,T] \times \mathbb{R}^d \times \mathbb{R}^d \rightarrow \mathbb{R}^d$ is the jump coefficient, $\bm{f}:[0,T] \times \mathbb{R}^m \times \mathbb{R}^{d} \times \mathbb{R}^{d\times m} \times \mathbb{R}^d \rightarrow \mathbb{R}^q$ is 
the generator of the BSDE, and the processes $\bm{\Gamma}_s: [0,T] \rightarrow \mathbb{R}^q$ 
is defined by $\bm{\Gamma}_s = \int_{E} \bm{U}_s(\bfe) \eta(\bfe) \lambda(d\bfe)$
for a given bounded function $\eta: \mathbb{R}^d \rightarrow \mathbb{R}$. The terminal condition $\bm{\xi}$, which is an $\mathcal{F}_T$-measurable random vector, is assumed to be a  function of $\bfX_T$, denoted by $\bm{\xi} = \bm{\varphi}(\bfX_T)$. 
Note that the integrals in \eqref{eq:BSDE} with respect to the $m$-dimensional Brownian motion $\bfW_t$ 
and the $d$-dimensional compensated Poisson measure $\tilde{\mu}(d\bfe, dt)$ are 
It\^{o}-type stochastic integrals. 
The first equation in \eqref{eq:BSDE} is the forward SDE system, which is a general L\'{e}vy process, while 
the second equation is a system of BSDEs driven by $\bfX_t$. The quadruplet $(\bfX_t, \bfY_t, \bfZ_t, \bm{U}_t)$ is 
called an $L^2$-adapted solution
if it is $\{\mathcal{F}_t\}$-adapted, square integrable process which satisfies the BSDE in \eqref{eq:BSDE}. 
Under standard assumptions on the given functions $\bm{b}$, $\bm{\sigma}$, $\bm{f}$, $\bm{\varphi}$ and $\bm{\beta}$, the existence and uniqueness of the solution of the system in \eqref{eq:BSDE}
with a nonlinear generator $\bm{f}$ have been proved in \cite{Anonymous:fk}. 

Based on the extension of the nonlinear Feynman-Kac formula for BSDEs proposed in \cite{Anonymous:fk}, 
the adapted solution $(\bfX_t, \bfY_t, \bfZ_t, \bm{U}_t)$ of \eqref{eq:BSDE} can be associated to the unique viscosity solution 
$\bfu(t,\bfx) \in \mathcal{C}([0,T]\times \mathbb{R}^d)$ of the backward Kolmogorov equation, i.e.,
\begin{equation}\label{eq:Kol}
\left\{
\begin{aligned}
&\frac{\partial \bfu}{\partial t}(t,\bfx) + \widetilde{\mathcal{L}}\bfu(t,\bfx) +\bm{f}(t,\bfx,\bfu,\bm{\sigma}\nabla \bfu, \mathcal{B}\bfu)=0, 
\text{ for } (t,\bfx) \in [0,T) \times \mathbb{R}^d\\
&\bfu(T,\bfx) = \bm{\varphi}(\bfx), \text{ for } \bfx \in \mathbb{R}^d.
\end{aligned}
\right.
\end{equation}
In \eqref{eq:Kol}, $\bm{\varphi}(\bfx)$ is the terminal condition at the time $t = T$, and the second-order integral-differential operator $\widetilde{\mathcal{L}}$ is of the form
\begin{equation}\label{eq:L1}
\begin{aligned}
& \widetilde{\mathcal{L}}\bfu(t,\bfx)  = \sum_{i=1}^d b_i(t,\bfx)\frac{\partial \bfu}{\partial x^i}(t,\bfx) + \frac{1}{2}\sum_{i,j = 1}^d (\bm{\sigma}\bm{\sigma}^{\top})_{i,j}(t,\bfx)
\frac{\partial^2 \bfu}{\partial x^i \partial x^j}(t,\bfx)\\
& \hspace{1.0cm} + \int_{E} \left(\bfu(t,\bfx + \bm{\beta}(t,\bfx,\bfe)) -\bfu(t,\bfx) - 
\sum_{i=1}^d \frac{\partial \bfu}{\partial x^i}(t,\bfx) \bm{\beta}(t,\bfx,\bfe) \right) \lambda (d\bfe),
\end{aligned}
\end{equation}
with $\mathcal{B}$ is an integral operator defined as 
\[
\mathcal{B} \bfu(t,\bfx) = \int_{E} \big[\bfu(t,\bfx + \bm{\beta}(t,\bfx,\bfe)) -\bfu(t,\bfx)\big] \eta(\bfe) \lambda(d\bfe). 
\]
The functions $\bm{b}$, $\bm{\sigma}$, $\bm{f}$, $\bm{\varphi}$, $\bm{\beta}$ in \eqref{eq:L1} have the same definitions as in the BSDE in \eqref{eq:BSDE}. Under the condition that $\bfX_{s} = \bfx$ for a fixed $s \in [0, T)$, 
the solution $(\bfY_t, \bfZ_t, \bm{U}_t)$ of the BSDE for $s \le t \le T$
can be represented by
\begin{equation}\label{FK}
\left\{
\begin{aligned}
&\bfY_t = \bfu(t,\bfX_t),\\
&\bfZ_t = \bm{\sigma}(t, \bfX_t) \nabla \bfu(t, \bfX_t), \\
& \bm{U}_t = \bfu(t,\bfX_{t-}+\bm{\beta}(t,\bfX_{t-},\bfe)) - \bfu(t, \bfX_{t-}),
\end{aligned}\right.
\end{equation}
where $\nabla \bfu$ denotes the gradient of $\bfu$ with respect to $\bfX_t$ and $ \bm{\Gamma}_t = \mathcal{B}\bfu(t, \bfX_t)$.
By comparing the operator $\widetilde{\mathcal{L}}$ in \eqref{eq:L1} and the BSDE in \eqref{eq:BSDE}, 
we can see that the incremental $d\bm{X}_t$ consists of three components: the drift term $\bm{b}(t,\bfX_t)dt$, 
the Brownian diffusion $\bm{\sigma}(t,\bfX_t)d\bfW_t$, and the 
jump diffusion $\int_{E} \bm{\beta}(t,\bfX_{t-},\bfe)\tilde{\mu}(d\bfe,dt)$.

Now let us relate the BSDE in \eqref{eq:BSDE} to the nonlocal diffusion equation in \eqref{nloc_diff}.
Without loss of generality, we only consider the case when $\bfY_t$ is a scalar function, i.e., by setting $q = 1$ in \eqref{eq:BSDE};  the 
numerical schemes and analysis presented in this paper can be directly extended to the case of $q > 1$. Observing that the operator $\widetilde{\mathcal{L}}$ in \eqref{eq:L1} is a generalization of the nonlocal operator $\mathcal{L}$ in \eqref{nloc_diff}, we aim to find a simplified form of the BSDE which is the stochastic representation of the nonlocal diffusion equation in \eqref{nloc_diff}. To proceed, we define $\lambda(d\bfe) = \gamma(\bfe)d\bfe$ for $\bfe \in E$ and $\lambda(d\bfe)= 0$ for $\bfe \notin E$, satisfying the condition in \eqref{lambda}.
Due to the integrability assumption of $\gamma(\bfe)$ in \eqref{gamma}, the jump diffusion term in 
\eqref{eq:BSDE} is a {\em compensated compound Poisson process}.
In this case, the compensated Poisson random measure 
$\tilde{\mu}(d\bfe, dt)$ can be represented by
\begin{equation}\label{mu}
\tilde{\mu}(d\bfe, dt) = \mu(d\bfe, dt) - \lambda \rho(\bfe) d\bfe dt,
\end{equation}
 where $\lambda(d \bfe) = \lambda \rho(\bfe)d\bfe$, with $\lambda$ being the intensity of Poisson jumps and $\rho(\bfe)d\bfe$ being the probability
measure of the jump size $\bm{\beta}(t,\bfx, \bfe)$ satisfying $\int_{E} \rho(\bfe) d\bfe =1$. Then, the operator $\widetilde{\mathcal{L}}$ in \eqref{eq:L1} can be simplified to the nonlocal operator ${\mathcal{L}}$ in \eqref{eq:L} by setting $m = 0$, i.e., removing the Brownian motion, and 
\begin{equation}\label{eq:coeff}
\begin{aligned}
&\bm{\sigma} = 0,\;\; \bm{\beta} = \bfe \,\mathcal{I}_{D}(\bfe),\;\; \lambda = \int_{E} \gamma(\bfe) d\bfe,\\
& \rho(\bfe) =\frac{1}{\lambda} \gamma(\bfe),\;\; b_i =  \int_{E} \bfe \gamma(\bfe) d\bfe \;\text{ for } i = 1, \ldots, d,
\end{aligned}
\end{equation}
where $\mathcal{I}_{D}(\bfe)$ is the characteristic function of the domain $D$. In the context of $\bm{\beta} = \bfe\,\mathcal{I}_{D}(\bfe)$, the drift coefficient $\bm{b}$ is a constant vector which indeed helps cancel the first-order derivatives of $u$
in $\widetilde{\mathcal{L}}$. On the other hand, substituting $\gamma(\bfe)$ into the definition of $\bm{b}$, we have
\begin{equation}\label{b}
\int_{0}^t b_i \; ds = \lambda \int_{0}^t \int_{E}  \bfe \rho(\bfe) d\bfe ds = \lambda t \mathbb{E}[\bfe]\; \text{ for } i = 1, \ldots, d,
\end{equation}
which is the compensator of $\bfX_t$. Hence, $\bfX_t$ is just a compound Poisson process under the Poisson measure $\mu(d\bfe,dt)$ 
without compensation. 
Then, by defining $f(t, \bfx, u) = g(T-t, \bfx, u)$ 
and $\varphi(\bfx) = u_0(\bfx)$, with $g$ and $u_0$ being the forcing term and initial condition in \eqref{nloc_diff} respectively,
we obtain the BSDE corresponding to the nonlocal diffusion equation in \eqref{nloc_diff}
\begin{equation}\label{eq:BSDE2}
\left\{
\begin{aligned}
\bm{X}_t&= \bm{X}_0 + \int_0^t \int_{E} \bm{e}\; \mu(d\bm{e},ds),\\
Y_t&= \varphi(\bm{X}_T)+ \int_t^T f(s, \bm{X}_s, Y_s) ds -  \int_t^T \int_{E} U_s(\bm{e})\; \tilde{\mu}(d\bm{e},ds).
\end{aligned}\right.
\end{equation}
According to Theorem 2.1 in \cite{Anonymous:fk}, the well-posedness of the BSDE in \eqref{eq:BSDE2} requires the generator $f$ is globally Lipschitz continuous, i.e., there exists $K>0$ such that
\[
|f(t,\bfx, y) - f(t, \bfx', y')| \le K(|\bfx - \bfx'| + |y-y'|)
\]
for all $0 \le t \le T$ with $\bfx, \bfx' \in \mathbb{R}^d$ and $y,y' \in \mathbb{R}$. 
In this case, 
there exists a unique solution 
$(Y_t, U_t) \in S^2 \times L^2(\tilde{\mu})$ where $S^2$ is the set of $\mathcal{F}_t$-adapted c\`{a}dl\`{a}g processes 
$\{Y_t, 0\le t \le T\}$ such that 
\[
\|Y_t\|_{S^2}^2 = \mathbb{E}\left[\left( \sup_{0\le t \le T} |Y_t| \right)^2 \right]< \infty,
\]
and $L^2(\tilde{\mu})$ is the set of mappings $U: \Omega \times [0,T] \times E \rightarrow \mathbb{R}$ such that 
\[
\| U_t\|_{L^2(\tilde{\mu})}^2 =  \mathbb{E} \left[\int_0^T \int_{E} U_t(\bfe)^2 \lambda(d\bfe) dt\right]  < \infty.
\]
To relate the solution $(Y_t, U_t)$ to the viscosity solution $u$ of the nonlocal diffusion equation in \eqref{nloc_diff}, we assume
$\varphi$ and $f$ are continuous functions satisfying
\[
|f(t,\bfx,0)| \le C (1+|\bfx|^p) \;\; \mbox{and}\;\; |\varphi(\bfx)| \le C (1+|\bfx|^p), \;\;(t,\bfx) \in [0,T]\times \mathbb{R}^d,
\]
for some real  $C, p >0$. Then, under the condition that $\bfX_s = \bfx$ for a fixed $s\in [0,T)$, the solution $(Y_t, U_t)$ for $s\le t \le T$ can be represented by
\[
Y_t = u(T-t, \bfX_t) \;\; \mbox{ and } \;\; U_t = u(T-t,\bfX_{t-}+\bfe) - u(T-t, \bfX_{t-}),
\]
where $u \in \mathcal{C}([0,T]\times \mathbb{R}^d)$ is the viscosity solution of \eqref{nloc_diff} satisfying
\[
|u(t,\bfx)| \le C(1+|\bfx|^p), \;\;(t,\bfx) \in [0,T]\times \mathbb{R}^d,
\]
again, for some real  $C, p >0$. Moreover, if $\varphi$ and $f$ are bounded and uniformly continuous, then $u$ is also bounded and uniformly continuous. 

Now we introduce the following notations which will be used throughout. 
Let $\mathcal{F}_s^{t,\bfx}$ $(t\leq s\leq T)$ be a $\sigma$-field generated by the stochastic process $\{\bfx + \bfX_r - \bfX_t, t \le r \le s \}$
starting from the time-space point $(t,\bfx)$, and set
$\mathcal{F}^{t,\bfx}=\mathcal{F}_T^{t,\bfx}$. Denote by $\mathbb{E}[\,\cdot\,]$
the mathematical expectation and
$\mathbb{E}_s^{t,\bfx}[\,\cdot\,]$ the conditional mathematical expectation under the $\sigma$-field
$\mathcal{F}_s^{t,\bfx}(t\leq s\leq T)$, i.e., 
$\mathbb{E}_s^{t,\bfx}[\,\cdot\,]=\mathbb{E}[\,\cdot\,|\mathcal{F}_s^{t,\bfx}]$.
Particularly, for the sake of notational convenience, when $s=t$, we use 
$\mathbb{E}_t^{\bfx}[\,\cdot\,]$ 
to denote $\mathbb{E}[\,\cdot\, |\mathcal{F}_t^{t,\bfx}]$, i.e., the mathematical expectation under the condition that $\bfX_t = \bfx$. 

\section{Numerical schemes for BSDEs with jumps}\label{sec:scheme}
In this section we propose numerical schemes for the BSDE in \eqref{eq:BSDE2} in order to solve the nonlocal diffusion equation in \eqref{nloc_diff}. Specifically, a semi-discrete scheme for time discretization is studied in \S \ref{sec:semi}, and a fully-discrete scheme is constructed in \S \ref{sec:fully} by incorporating appropriate spatial discretization techniques. 

\subsection{The semi-discrete scheme}\label{sec:semi}
For the time interval $[0,T]$, we introduce the partition
\begin{equation}\label{dt}
\mathcal{T}_\tau =  \big\{ 0=t_0<\cdots<t_N=T \big\}
\end{equation}
with $\Delta t_n = t_{n+1} - t_n$ and $\tau = \max_{0 \le n\le N-1} \Delta t_n$.  
In the time interval $[t_n, t_{n+1}]$ for $0\leq n\leq N-1$, under the condition that
$\bfX_{t_n} = \bfx \in \mathbb{R}^d$, the BSDE in \eqref{eq:BSDE2} can be rewritten as
\begin{equation}\label{s3:e1}
\hspace{-2.1cm}\bfX_s = \bfx + \int_{t_n}^s \int_{E} \bm{e}\; \mu(d\bm{e},dr)\; \text{ for } \; t_n \le s \le t_{n+1},
\end{equation}
\vspace{-0.5cm}
\begin{equation}\label{s3:e1_2}
Y_{t_n}
=Y_{t_{n+1}}+\int_{t_n}^{t_{n+1}} f(s,\bfX_s, Y_s) ds
-\int_{t_n}^{t_{n+1}} \int_{E} U_s(\bfe)\tilde{\mu}(d\bfe,ds).
\end{equation}
Taking the conditional mathematical expectation
$\mathbb{E}_{t_n}^{\bfx}[\,\cdot\,]$ on both sides of \eqref{s3:e1_2}, due to the fact that 
$\int_{t_n}^{t} \int_{E} U_s(\bfe)\tilde{\mu}(d\bfe,ds)$ for $t > t_{n}$ is a martingale, we obtain that 
\begin{equation}
\label{s3:e2}
Y_{t_n}=\mathbb{E}_{t_n}^{\bm{x}}[Y_{t_{n+1}}]+\int_{t_n}^{t_{n+1}}\mathbb{E}_{t_n}^{\bfx}[f(s, \bfX_s, Y_s)]ds,
\end{equation}
where $\mathbb{E}_{t_n}^{\bfx}[Y_{t_n}] = Y_{t_n}(\bfx) = u(T-t_n,\bfx)$, $Y_{t_{n+1}}= Y_{t_{n+1}}(\bfX_{t_{n+1}}) = u(T-t_{n+1}, \bfX_{t_{n+1}})$ and the integrand $\mathbb{E}_{t_n}^{\bfx}[f(s, \bfX_s, Y_s)]$ is a deterministic function of $s \in [t_n, t_{n+1}]$. 
To estimate the integral in \eqref{s3:e2}, several numerical methods for approximating integrals with deterministic integrands can be used. Here, we utilize the $\theta$-scheme proposed 
in \cite{Zhao:2006jx, Zhao:2012ku}, which yields
\begin{equation}\label{s3:e3}
\begin{aligned}
Y_{t_n}=\mathbb{E}_{t_n}^{\bfx}[Y_{t_{n+1}}]
& +(1-\theta)\Delta t_n \mathbb{E}_{t_n}^{\bfx}[f(t_{n+1},\bfX_{t_{n+1}}, Y_{t_{n+1}})]\\
& +\theta\Delta t_n f(t_n, \bfX_{t_n},Y_{t_n}) +R_n,
\end{aligned}
\end{equation}
where $\theta \in [0,1]$ and the residual $R_n$ is given by
\begin{equation}\label{eq:Ry}
\begin{aligned}
R_n=\int_{t_n}^{t_{n+1}}& \Big\{\mathbb{E}_{t_n}^{\bfx}[f(s,\bfX_s,Y_s)]  -\theta
f(t_n, \bfX_{t_n},Y_{t_n})\\
& \quad -(1-\theta)\mathbb{E}_{t_n}^{\bfx}[f(t_{n+1}, \bfX_{t_{n+1}},Y_{t_{n+1}})]\Big\}ds.
\end{aligned}
\end{equation}
Next we propose a semi-discrete scheme for the BSDE in \eqref{eq:BSDE2} as follows: given the random variable $Y_N = Y_{t_N} = \varphi(\bfX_T)$ as the terminal condition, for $n = N-1, \ldots, 1,0$, under the condition that $\bfX_n = \bfx$, the solution $Y_{t_n}$ is approximated by $Y_n$ satisfying
\begin{equation}\label{eq:semi}
\left\{
\begin{aligned}
& \bfX_{n+1}  = \bfX_n + \int_{E} \bfe \; {\mu}(d\bfe, \Delta t),\\
& Y_n  = \mathbb{E}_{t_n}^{\bfx}\big[Y_{n+1}\big] + (1-\theta)\Delta t_n \mathbb{E}_{t_n}^{\bfx}
\big[f(t_{n+1}, \bfX_{n+1}, Y_{n+1})\big]\\
&\hspace{2.6cm}+\theta\Delta t_n f(t_n, \bfX_n, Y_n),\\
\end{aligned}\right.
\end{equation}
where $0 \le \theta \le 1$. By choosing different values for $\theta$, 
we obtain different semi-discrete schemes, 
e.g., $\theta = 0$, $\theta=1$ and $\theta = \frac{1}{2}$ lead to forward Euler (FE), backward Euler (BE) 
and Crank-Nicolson (CN) schemes, respectively.

{Note that since $\bfX_t$ is the standard compound Poisson process, the probability distributions of $\bfX_t$ or any incremental $\bfX_t - \bfX_{t'}$, with $0 \le t' < t \le T$, are well known. Thus, one does not need to
discretize $\bfX_{t}$ so that, in \eqref{eq:semi}, $\bfX_{n+1}$ denotes the exact solution evaluated at $t_{n+1}$. In this case, $R_n$ in \eqref{eq:Ry} is the local truncation error of the scheme \eqref{eq:semi}; a rigorous error analysis of $R_n$ is given in \S\ref{sec:err}. Other time-stepping schemes, e.g., linear multi-step schemes \cite{Zhao:2010ik}, can be directly used in order to further improve the accuracy of time discretization. 
Note that in the general case where the coefficients $\bm{b}$ and $\bm{\sigma}$ in \eqref{eq:BSDE} and \eqref{eq:L1} are functions of
$t$ and $\bfX_t$, another time-stepping scheme \cite{Platen:2010eo} is also needed for to discretize $\bfX_t$, 
so that an extra local truncation error will be introduced into the semi-discrete scheme \eqref{eq:semi}. This is beyond the scope of this effort but will be the focus of future work.}

\subsection{The fully-discrete scheme}\label{sec:fully}
The semi-discrete scheme \eqref{eq:semi} is not computationally feasible, because the involved conditional mathematical expectation $\mathbb{E}_{t_n}^{\bm x}[\cdot]$ is defined over the whole real space $\mathbb{R}^d$. Thus, an effective spatial discretization approach 
is also necessary in order to approximate $\mathbb{E}_{t_n}^{\bfx}[\,\cdot\,]$. 
To proceed, we partition the $d$-dimensional Euclidean space $\mathbb{R}^d$ by
\[
\mathcal{S}^d_h = \mathcal{S}^1_{h_1} \times \mathcal{S}^2_{h_2}  \times \cdots \times \mathcal{S}^d_{h_d} ,
\]
where $\mathcal{S}_{h_k}^k$ for $k = 1, \ldots, d$ is a partition of the one-dimensional real space $\mathbb{R}$, i.e.,
\[
\mathcal{S}_{h_k}^k= \Big\{ x_i^k \;\Big|\;  x_i^k \in \mathbb{R}, i \in \mathbb{Z}, x_i^k < x^k_{i+1}, 
\lim_{i\rightarrow +\infty} x_i^k = +\infty, \lim_{i\rightarrow -\infty} x_i^k = -\infty \Big\},
\]
such that $h_k = \max_{i \in \mathbb{Z}}\{|x_i^k - x_{i-1}^k|\}$ and $h = \max_{1\le k \le d} h_k$. For each multi-index
$\ii = (i_1, i_2, \ldots, i_d) \in \mathbb{Z}^d$, the corresponding grid point in $\mathcal{S}_h^d$ is denoted by $\bfx_{\ii} = (x^1_{i_1}, \ldots, x^d_{i_d})$.

Now we turn to constructing a quadrature rule for approximating $\mathbb{E}_{t_n}^{\bfx_{\ii}}[\,\cdot\,]$ in \eqref{eq:semi} for $(t_n, \bfx_{\ii}) \in \mathcal{T}_\tau \times \mathcal{S}_h^d$.
Such expectation is defined with respect to the probability distribution of the incremental stochastic process
$\Delta \bfX_{n+1} = \bfX_{n+1} - \bfX_{n} =  \bfX_{n+1}  - \bfx_{\ii}$, which is a compound Poisson process starting from the grid point $(t_n,\bfx_{\ii})$. 
Here we take $\mathbb{E}_{t_n}^{\bfx_{\ii}}[Y_{n+1}]$ as an example and the proposed quadrature rule can be directly extended to approximate 
$\mathbb{E}_{t_n}^{\bfx_{\ii}}[f(t_{n+1},\bfX_{n+1}, Y_{n+1})]$. 
It is well known that the number of jumps of $\bfX_t$ within $(t_n, t_{n+1}]$, denoted by $N_t$, follows a Poisson distribution with intensity $0<\lambda < +\infty$, 
and the size of each Poisson jump follows the distribution $\rho(\bfe)d\bfe$ defined in \eqref{eq:coeff}. Thus, $\mathbb{E}_{t_n}^{\bfx_{\ii}}[Y_{n+1}]$ can be represented by
\begin{equation}\label{eq:EY}
\begin{aligned}
\mathbb{E}_{t_n}^{\bfx_{\ii}} \big[Y_{n+1}\big] 
& = \sum_{m=0}^\infty \mathbb{P}\Big\{N_{t_{n+1}} - N_{t_n} =m\Big\}\; \mathbb{E} \bigg[Y_{n+1}\bigg(\bfx_{\ii} + \sum_{k=1}^m \bfe_k\bigg)\bigg]\\
& = \sum_{m=0}^\infty\exp(-\lambda\Delta t_n) \; \frac{(\lambda \Delta t_n)^m }{m!}\;  \mathbb{E} \bigg[Y_{n+1}\bigg( \bfx_{\ii} + \sum_{k=1}^m \bfe_k\bigg)\bigg]\\
& = \exp(-\lambda\Delta t_n) \; Y_{n+1}(\bfx_{\ii}) +  \sum_{m=1}^\infty  \exp(-\lambda\Delta t_n) \; \frac{(\lambda \Delta t_n)^m }{m!}
\\
& \hspace{0.5cm}\times \int_{E} \cdots \int_{E} Y_{n+1}\bigg(\bfx_{\ii} + \sum_{k=1}^m \bfe_k\bigg)
\bigg(\prod_{k=1}^m \rho(\bfe_k) \bigg)d\bfe_1\cdots d\bfe_m,\\
\end{aligned}
\end{equation}   
where $\bfe_k = (e_k^1,\ldots, e_k^d)$ for $k = 1, \ldots, m$ is the size of the $k$-th jump. 
Observing that
the probability of having $m$ jumps within $(t_n, t_{n+1}]$ 
is of order $\mathcal{O}((\lambda \Delta t_n)^m)$, the conditional mathematical expectation $\mathbb{E}_{t_n}^{\bfx_i}[Y_{n+1}]$ 
can be approximated by a truncation of \eqref{eq:EY}, i.e., the sum of a {\em finite} sequence, where the 
number of terms retained is determined according to the prescribed accuracy. 
For example, if we take $\theta =  \frac{1}{2}$ in \eqref{eq:semi}, 
we expect a second-order convergence from the time discretization, 
i.e., the local truncation error $R_n$ in \eqref{s3:e3} should be of order $\mathcal{O}((\Delta t_n)^3)$.
In this case, assuming the intensity value $\lambda$ is of order $\mathcal{O}(1)$, the first three terms should be retained in
\eqref{eq:EY} in order to guarantee the error from the truncation of \eqref{eq:EY} matches the order of $R_n$. 
Hence, in the sequel, we denote by ${\mathbb{E}}_{t_n,M_y}^{\bfx_{\ii}}[Y_{n+1}]$ the approximation of $\mathbb{E}_{t_n}^{\bfx_{\ii}}[Y_{n+1}]$ by retaining the first $M_y+1$ terms, where $M_y$ indicates the number of jumps included in ${\mathbb{E}}_{t_n,M_y}^{\bfx_{\ii}}[Y_{n+1}]$. An analogous notation ${\mathbb{E}}_{t_n,M_f}^{\bfx_{\ii}}[f_{n+1}]$ is used to represent the approximation of ${\mathbb{E}}_{t_n}^{\bfx_{\ii}}[f_{n+1}]$ by retaining $M_f$ jumps.

Next, we also need to approximate the $m\times d$-dimensional integral in \eqref{eq:EY} for $m = 1, \ldots, M_y$. 
This can be accomplished by selecting an appropriate quadrature rule base on 
the properties of $\rho(\bfe)$ and the smoothness of $Y_{n+1}$. A straightforward choice is to utilize Monte Carlo methods by drawing samples from $\prod_{k=1}^m \rho(\bfe_k)$, but this is inefficient because of slow convergence. To enhance the convergence rate, an alternative way is to use the tensor product of high-order one-dimensional quadrature rules, e.g., Newton-Cotes, Gaussian, etc.. In addition, for multi-dimensional problems, i.e., $d > 1$, by requiring second-order convergence $\mathcal{O}((\Delta t)^2)$, i.e., $M_y = 2$, sparse-grid quadrature rules \cite{Bungartz:2004kx,Gunzburger:2014hi} can be applied to further improve the computational efficiency.
Without loss of generality, for $m = 1, \ldots, M_y$, we denote by $\{w_{q}^m\}_{q=1}^{Q_m}$ and $\{\bm{a}_q^{1,m}, \ldots, \bm a_q^{m,m}\}_{q=1}^{Q_m}$ the set of weights and points, respectively, of the selected quadrature rule for estimating the $m$-th integral in \eqref{eq:EY}. 
%
Then, the approximation of 
$\mathbb{E}_{t_n}^{\bfx_{\ii}}[Y_{n+1}]$, denoted by $\widehat{\mathbb{E}}_{t_n,M_y}^{\bfx_{\ii}}[Y_{n+1}]$ is given by
\begin{equation}\label{appro_EY}
\begin{aligned}
\widehat{\mathbb{E}}_{t_n,M_y}^{\bfx_{\ii}}\big[Y_{n+1}\big] & = \exp(-\lambda\Delta t) \; Y_{n+1}(\bfx_{\ii}) \\
& \quad + \sum_{m=1}^{M_y} \exp(-\lambda \Delta t) \frac{(\lambda \Delta t)^m}{m!} \; \sum_{q=1}^{Q_m} 
w^m_q \;  Y_{n+1}\Big( \bfx_{\ii} + \sum_{k=1}^m\bm{a}^{k,m}_{q}\Big).
\end{aligned}
\end{equation}

Note that it is possible that the quadrature points $\{\bfx_{\ii}+ \sum_{k=1}^m\bm{a}_{q}^{k,m}\}_{q = 1}^{Q_m}$ in \eqref{appro_EY}
do not belong to the spatial grid $\mathcal{S}_h^d$.  The simplest approach to manage this difficult is to add all quadrature points to the spatial grid 
at each time stage. However, this will result in an exponential growth of the total number of grid points as the number of time steps increases. 
Thus, we follow the same strategy as in \cite{Zhao:2006jx,Zhao:2012ku,Zhao:2010ik}, and construct piecewise Lagrange interpolating polynomials based on $\mathcal{S}_h$ to approximate $Y_{n+1}$ at the non-grid points. Specifically, at any given point $\bfx = (x^1, \ldots, x^d) 
\in \mathbb{R}^d$, $Y_{n+1}(\bfx)$ is approximated by
\[
Y_{n+1}(\bfx) \approx {Y}_{n+1,p}(\bfx) = \sum_{j_1=1}^{p+1} \cdots \sum_{j_d=1}^{p+1} \Bigg({Y}_{n+1,p}^{i_{j_1}, \ldots, i_{j_d}}
\; \prod_{k = 1}^d \prod_{1\le j \le p+1\atop j\neq j_k} \frac{x^k-x_{i_{j}}^k}{x^k_{i_{j_k}}-x^k_{i_{j}}}\Bigg), 
\]
where ${Y}_{n+1,p}(\bfx)$ is a $p$-th order tensor-product Lagrange interpolating polynomial. 
For $k = 1,\ldots, d$, the interpolation points $\{x^k_{i_{j}}\}_{j=1}^{p+1} \subset \mathcal{S}_h^k$ are the closest $p+1$ neighboring points of $x^k$ such that $(x^1_{i_{j_1}}, \ldots, x^d_{i_{j_d}})$, for $j_k = 1, \ldots, p+1$ and $k = 1, \ldots, d$, constitute a local tensor-product 
sub-grid around $\bfx$. 
The fully-discrete solution at the interpolation point 
$(x^1_{i_{j_1}}, \ldots, x^d_{i_{j_d}}) \in \mathcal{S}_h^d$ is given by 
${Y}_{n+1,p}^{{i_{j_1}, \ldots, i_{j_d}}}$.
Note that ${Y}_{n+1,p}(\bfx)$ does not interpolate the semi-discrete solution $Y_{n+1}(\bfx)$ 
due to the error $Y_{n+1}(\bfx_{\ii}) - {Y}_{n+1,p}^{\ii}$. 
 Therefore, the fully-discrete scheme of the BSDE in \eqref{eq:BSDE2} is described as follows: given 
the random variable $Y_{N}(\bfx_{\ii})$ for $\ii \in \mathbb{Z}^d$, and for $n=N-1,\ldots, 0$, solve the quantities ${Y}_{n,p}^{\ii}$ for $\ii \in \mathbb{Z}^d$ such that
\begin{equation}\label{eq:full}
\left\{
\begin{aligned}
& \bfX_{n+1}  = \bfx_{\ii} + \int_{E} \bfe \; {\mu}(d\bfe, \Delta t),\\
& {Y}_{n,p}^{\ii}= \widehat{\mathbb{E}}_{t_n,M_y}^{\bfx_{\ii}}\left[{Y}_{n+1,p}\right] +
 (1-\theta)\Delta t_n \widehat{\mathbb{E}}_{t_n,M_f}^{\bfx_{\ii}}\left[f(t_{n+1}, \bfX_{n+1}, {Y}_{n+1,p})\right]\\
&  \hspace{3.5cm}+\theta\Delta t_n f\left(t_n, \bfX_{n}, {Y}_{n,p}^{\ii}\right),\\
\end{aligned}\right.
\end{equation}
where $\theta \in [0,1]$ and the non-negative integers $M_y$, $M_f$ indicate the Poisson jumps included in the approximations of $\mathbb{E}_{t_n}^{\bfx_\ii}[Y_{n+1,p}]$ and $\mathbb{E}_{t_n}^{\bfx_\ii}[f]$, respectively. 

Note that, by using the quadrature rule, the approximation in \eqref{appro_EY} is defined in a bounded domain for a fixed $\bm x_\ii \in \mathcal{S}^d_h$. As such, the approximate solution $Y_{0,p}$ in any bounded subdomain of $\mathbb{R}^d$ can be obtained by solving the fully-discrete scheme \eqref{eq:full} in a bounded subset of $\mathcal{S}_h^d$. 
We also observe from \eqref{eq:full} that at each point $(t_n, x_{\ii})\in \mathcal{T}_\tau \times \mathcal{S}_h^d$, the computation of $Y_{n,p}^{\ii}$ only depends on $\bfX_{n+1}$ and $Y_{n+1,p}$ even though an implicit time-stepping scheme ($\theta > 0$) is used. This means 
$\{Y_{n,p}^{\ii}\}_{\ii \in \mathbb{Z}}$ on each time stage can be computed independently, so that
the difficulty of solving linear systems with possibly dense matrices is completely avoided. Instead, $Y_{n,p}^{\ii}$ 
can be either computed explicitly (in  case $f$ is linearly dependent on $Y_{n,p}^{\ii}$), 
or obtained by solving a nonlinear equation (in  case $f$ is nonlinear with respect to $Y_{n,p}^{\ii}$). 
This feature makes it straightforward to develop massively parallel algorithms, 
which are, in terms of scalability, very similar to the CTRW method.
Moreover, the scheme \eqref{eq:full} is expected to outperform the CRTW method because it achieves high-order accuracy 
(discussed in \S\ref{sec:err}) from the discretization and is also capable of solving problems that include a general nonlinear forcing term $g$. 
In fact, the combination of accurate approximations and scalable computations are the key advantages of our 
approach, compared to the existing methods for the nonlocal diffusion problem in \eqref{nloc_diff}. 

\begin{rem}
The nonlinear Feynman-Kac
formula converted the effect of the nonlocal operator $\mathcal{L}$ in \eqref{eq:L} to the behavior of the corresponding 
compound Poisson process, which is accurately approximated by the conditional mathematical expectations 
$\mathbb{E}_{t_n}^{\bfx}[\,\cdot\,]$.
As such, there is no need to discretize the nonlocal operator in the fully-discrete scheme \eqref{eq:full}, so that stability does not require any CFL-type restriction on the time step; this is in contrast to classic numerical schemes for which explicit schemes do require a CFL condition.  Moreover, according to the 
analysis in \cite{Zhao:2006jx,Zhao:2010ik}, the approximation \eqref{eq:full} is stable as long as the semi-discrete scheme is absolutely stable.
\end{rem}

\begin{rem}
The total computational cost of the scheme \eqref{eq:full} is dominated by the cost of approximating
$\mathbb{E}_{t_n}^{\bfx_\ii}[\,\cdot\,]$ at each grid point $\bfx_{\ii} \in \mathcal{S}_h^d$, using the formula in \eqref{appro_EY}.
For example, when solving a three-dimensional problem ($d=3$) and retaining two L\`{e}vy jumps ($M_y = M_f =2$,) our approach requires the approximation of multiple  
six-dimensional integrals. In this case, sparse-grid quadrature rules \cite{Zhang:2013en} can be used to alleviate the explosion in computational cost coming from the high-dimensional systems.
\end{rem}

%

%
%
\section{Error estimates}\label{sec:err}
In this section, we analyze both the semi-discrete and fully-discrete schemes, given by 
\eqref{eq:semi} and \eqref{eq:full} respectively, 
for the BSDE in \eqref{eq:BSDE2}. Without
loss of generality, we only consider the one-dimensional case ($d = 1$) and set $\theta = \frac{1}{2}$;
the following analysis can be directly extended to multi-dimensional cases. For simplicity, we also assume
$\mathcal{T}_\tau$ and $\mathcal{S}_h^d$ are both uniform grids, i.e., $\Delta t = \Delta t_n$ 
for $n = 1, \ldots, N$ and $\Delta x = x_i - x_{i-1}$ for $i \in \mathbb{Z}$.

Note that the well-posedness of the BSDE in \eqref{eq:BSDE2} only requires global Lipchitz continuity on $f$ with respect to $X_t$ and $Y_t$. However, 
in order to obtain error estimates, we need to impose stronger regularity on $f$. To this end, we first introduce the following notation:
\[
\begin{aligned}
& \mathcal{C}^{k}_b \left( D_1 \times \cdots \times D_J \right)\\
 = & \bigg\{\psi: \prod_{j=1}^J D_j\rightarrow \mathbb{R} \;\bigg|\;  \dfrac{\partial^{|\bm{\alpha}|}\psi}{\partial \bfx^{\bm{\alpha}}} \text{  is bounded and continuous} \mbox{ for } 0\le |\bm{\alpha}|\le k
\\
&  \quad\quad \mbox{ where } \bm{\alpha} = (\alpha_1, \dots,\alpha_J) \mbox{ with } \{\alpha_j\}_{j=1}^J \subset \mathbb{N}
\mbox{ and } |\bm{\alpha}| = \alpha_1 + \cdots + \alpha_J \bigg\},
\end{aligned}
\]
where $D_1 \times \cdots \times D_J \subset \mathbb{R}^J$.
Next we present the following lemma that relates the smoothness of $f(t,x,y)$ to 
the regularity of the solution $u(t,x)$ of the nonlocal diffusion equation in \eqref{nloc_diff}.
\begin{lem}\label{lem1}
Under assumptions in \eqref{gamma} and \eqref{lambda}, if $\varphi(x)\in\mathcal C_b^k(\mathbb{R})$,
$f(t,x,y) \in \mathcal C_b^k([0,T] \times \mathbb{R}^2)$ in \eqref{eq:BSDE2} and $\frac{\partial^k f}{\partial  x^{\alpha_1}\partial y^{\alpha_2}}$ with $\alpha_1+\alpha_2\le  k$ is uniformly Lipschitz 
continuous with respect to $x$ and $y$, then the nonlocal diffusion equation  
in \eqref{nloc_diff} with $d=1$, $u_0(x) = \varphi'(x)$ and $g(t, x, u) = f(T-t,x, u)$ has a unique viscosity solution 
$u(t,x) \in\mathcal{C}_b^{k}([0,T]\times \mathbb{R})$.
\end{lem}
\begin{proof} 
Here we only prove the case of $k=1$, as the cases $k>1$ can be proved by recursively using the following argument. 
Based on the regularity of $f$, it is easy to see that $\frac{\partial u}{\partial t}$ is bounded and continuous with respect to $t$, so that we focus on proving the regularity of $\frac{\partial u}{\partial x}$. 
Let $\nabla X_t$, $\nabla Y_t$ and $\nabla U_t$ be the variations of 
the stochastic processes $X_t$, $Y_t$ and $U_t$ with respect to $X_t$
in \eqref{eq:BSDE2}, respectively.  
Then the triple $(\nabla X_t, \nabla Y_t,\nabla U_t)$ is the solution of the following BSDE.
\begin{equation}\label{eq:V_BSDEaa}
\left\{
\begin{aligned}
\nabla {X}_t= & 1,\\
\nabla Y_t=& \varphi'({X}_T)+ \int_t^T f_x'(s, {X}_s, Y_s) 
+f_y'(s, {X}_s, Y_s) \nabla Y_s ds\\
& -  \int_t^T \int_{E} \nabla U_s({e})\; \tilde{\mu}(d{e},ds),
\end{aligned}\right.
\end{equation}
which is a linear equation with respect to 
$\nabla X_s$, $\nabla Y_s$ and $\nabla U_s$. Due to the regularity of $f$, the BSDE in \eqref{eq:V_BSDEaa}
has a unique  solution $(\nabla Y_t, \nabla U_t)$ where $\nabla Y_t$ is bounded and continuous. By the relationship between BSDEs and PIDEs, we see that $\nabla Y_t = w(t,X_t)$ where $w(t,x)$ is the unique viscosity solution of the following PIDE \cite{Anonymous:fk}
\begin{equation}\label{V_nloc_diff}
\left\{
\begin{aligned}
& \frac{\partial w}{\partial t}(t,x) + \mathcal{L} w(t,x) 
+ f_x'(t,x,v)+f_v'(t,x,v) w=0 & (t,x) \in [0,T] \times \mathbb{R},\\
& w(T,x) = \varphi'(x), \text{ for } x \in \mathbb{R},
\end{aligned}\right.
\end{equation}
where $v = u(T-t,x)$ is the transformed solution of the nonlocal problem in \eqref{nloc_diff} with $d=1$, $u_0 = \varphi'(x)$ and $g(t, x, u) = f(T-t,x, u)$. Then, by differentiating the equations in \eqref{nloc_diff} with respect to $x$ and comparing the resulting equations with \eqref{V_nloc_diff}, we can see that $w(t,x) = \frac{\partial u }{\partial x} (T-t,x)$. 
Due to the continuity and boundedness of $w(t,x)$, we have that $u(t,x) \in \mathcal{C}_1^b([0,T]\times \mathbb{R})$,
which completes the proof.
\end{proof}

Next we provide the estimate of the local truncation error $R_n$ in \eqref{eq:Ry}.

\begin{thm}\label{err_lem3} Using Lemma \ref{lem1} with $k=2$,
 the local truncation error $R_n$ in \eqref{s3:e3}, for the semi-discrete scheme \eqref{eq:semi} with $\theta = \frac{1}{2}$ can be bounded by
\begin{equation}
 \mathbb{E}[|R_n|]\leq C(\Delta t)^3 \; \text{ for }\;  n = 0, \ldots, N-1,
\end{equation}
where the constant $C$ depends only on the terminal time $T$, the jump intensity $\lambda$, the upper bounds
of $f$, $\varphi$, and their derivatives. 
\end{thm}

\begin{proof}
The formula \eqref{FK} shows that the solution $Y_t$ of the BSDE \eqref{eq:BSDE2} can be represented by
$Y_t =u(t,X_t)$. Under Lemma \ref{lem1}, if $f\in \mathcal{C}_b^{2}([0,T]\times \mathbb{R}^2)$ and $\varphi \in \mathcal{C}_b^{2}(\mathbb{R})$, then $u(t,X_t) \in \mathcal{C}_b^{2}([0,T] \times \mathbb{R})$. Thus, by defining $F(t,X_t) = f(t,X_t, u(T-t,X_t))$, we have $F\in \mathcal{C}_b^{2}([0,T]\times \mathbb{R})$. Before estimating the residual $R_n$ in \eqref{s3:e3},
we define the  differential operators $L^0$ and $L^1$ by
\begin{equation}
\left\{
\begin{aligned}
&L^0 F(t,X_t) = \frac{\partial F}{\partial t}(t,X_t) + b\frac{\partial F}{\partial x}(t,X_t) \\
& \hspace{2cm} + \int_{E} \bigg[ F(t,X_{t-}+e) - F(t,X_{t-}) - \frac{\partial F}{\partial x}(t,X_{t-}) e \bigg] \lambda(de), \\
&L^1 F(t,X_{t}) = F(t,X_{t-}+e) - F(t,X_{t-}).
\end{aligned}\right.
\end{equation}
Based on the definition of $X_s$ in \eqref{eq:BSDE2}, 
the integral form of the It\^{o} formula of $F(s, X_s)$ for $t_n \le s\le t_{n+1}$, under the condition $X_{t_n} = x$, is given by
\begin{equation} 
\begin{aligned}
F(s,X_s)
&=F(t_n,x)+\int_{t_n}^s L^0F(r,X_r)dr
+\int_{t_n}^s \int_{E} L^{1}F(r,X_{r-})\tilde{\mu}(de,dr).
\end{aligned}
\end{equation}
Thus, substituting the above formula into $\int_{t_n}^{t_{n+1}} \mathbb{E}_{t_n}^x [F(s, X_s)] ds$, we obtain

\begin{equation}
\begin{aligned}
&\int_{t_n}^{t_{n+1}} \mathbb{E}_{t_n}^x [F(s, X_s)] ds\\
= &\int_{t_n}^{t_{n+1}}
\mathbb{E}_{t_n}^x \left[F(t_n,x)  + \int_{t_n}^s L^0 F(r,X_r)  dr + \int_{t_n}^s L^1 F(r,X_{r-}) \tilde{\mu}(de,dr) \right] ds\\
= & F(t_n,x) \int_{t_n}^{t_{n+1}} ds + \int_{t_n}^{t_{n+1}}\int_{t_n}^s \mathbb{E}_{t_n}^x [L^0 F(r,X_r)] dr ds\\
= & F(t_n,x)\Delta t +  \int_{t_n}^{t_{n+1}}\int_{t_n}^s 
\mathbb{E}_{t_n}^x\left[ L^0 F(t_n,x) +   \int_{t_n}^r L^0L^0 F(z,X_z)  dz\right] dr ds\\
= & F(t_n,x)\Delta t + \frac{1}{2} (\Delta t)^2 L^0 F(t_n,x) +  \int_{t_n}^{t_{n+1}}\int_{t_n}^s \int_{t_n}^r 
\mathbb{E}_{t_n}^x [L^0L^0 F(z,X_z)]  dzdrds.
\end{aligned}
\end{equation}
Similar derivation can be conducted to $F(t_{n+1}, X_{t_{n+1}})$ to obtain
\begin{equation}
\begin{aligned}
 &\int_{t_n}^{t_{n+1}} \mathbb{E}_{t_n}^x [F(t_{n+1}, X_{t_{n+1}})] ds\\
= & F(t_n,x)\Delta t +  (\Delta t)^2 L^0 F(t_n,x) +  \int_{t_n}^{t_{n+1}}\int_{t_n}^{t_{n+1}} \int_{t_n}^r 
\mathbb{E}_{t_n}^x [L^0L^0 F(z,X_z)]  dzdrds.
\end{aligned}
\end{equation}
Therefore, the residual $R_n$ in \eqref{s3:e3} with $\theta = \frac{1}{2}$ can be represented by
\begin{equation}
\begin{aligned}
R_n & = \int_{t_n}^{t_{n+1}}\mathbb{E}_{t_n}^x[ F(s,X_s)]-\frac{1}{2}
F(t_n, x)-\frac{1}{2}\mathbb{E}_{t_n}^x[F(t_{n+1},X_{t_{n+1}})]ds \\
 & = \int_{t_n}^{t_{n+1}}\int_{t_n}^s \int_{t_n}^r \mathbb{E}_{t_n}^x [L^0L^0 F(z,X_z)]  dzdrds\\
 & \hspace{1.0cm}- \frac{1}{2} \int_{t_n}^{t_{n+1}}\int_{t_n}^{t_{n+1}} \int_{t_n}^r \mathbb{E}_{t_n}^x [L^0L^0 F(z,X_z)]  dzdrds.
\end{aligned}
\end{equation}
Then, taking the mathematical expectation of $R_n$, and using Cauchy-Schwarz inequality, we have that
\begin{equation}
\label{eq:CS1}
\begin{aligned}
\mathbb{E}[|R_n|] & \le \mathbb{E}\left[\left|\int_{t_n}^{t_{n+1}}\int_{t_n}^{s}\int_{t_n}^{r}
\mathbb{E}_{t_n}^x[L^0L^0F(z,X_{z})]dzdrds\right|\right] \\
& \quad +\frac{1}{2}\mathbb{E}\left[\left|\int_{t_n}^{t_{n+1}}
\int_{t_n}^{t_{n+1}}\int_{t_n}^{r}
\mathbb{E}_{t_n}^x[L^0 L^0F(z,X_{z})]dzdrds\right|\right]\\
& \le (\Delta t)^{\frac{3}{2}} \sqrt{\int_{t_n}^{t_{n+1}}\int_{t_n}^{s}\int_{t_n}^{r}
\mathbb{E}[|L^0L^0F(z,X_{z})|^2]dzdrds}\\
& \quad + \frac{1}{2}(\Delta t)^{\frac{3}{2}} \sqrt{\int_{t_n}^{t_{n+1}}
\int_{t_n}^{t_{n+1}}\int_{t_n}^{r}
\mathbb{E}[|L^0 L^0F(z,X_{z})|^2]dzdrds}\\
&\le \frac{3}{2} \sup_{0\le t \le T} \sqrt{\mathbb{E}[|L^0 L^0F(t,X_{t})|^2]} (\Delta t)^3\\
& \le C (\Delta t)^3,
\end{aligned}
\end{equation}
where, by the definition of $L^0$, it is easy to see that the constant $C$ only depends on the terminal time $T$, the jump intensity $\lambda$ the upper bounds of 
$f$, $\varphi$ and their derivatives. 
\end{proof}

We now turn our attention to proving an error estimate for semi-discrete scheme \eqref{eq:semi}. 
\begin{thm}\label{TL}
Let $Y_{t_n}$ and $Y_n$, for $n=0,1,\cdots,N$, be the exact solution of the BSDE \eqref{eq:BSDE2}
and the semi-discrete solution obtained by the scheme \eqref{eq:semi}, respectively.
Then, under the same conditions as that for Theorem \ref{err_lem3}, for sufficiently small time step $\Delta t$, the global truncation
error $e_n = Y_{t_n} - Y_n$, for $n=N-1,\dots,0$, can be bounded by
\begin{equation}\label{TL:e1}
\mathbb{E}[|Y_{t_n} - Y_n|] \le C (\Delta t)^2,
\end{equation}
where the constant $C$ is the same as in Theorem \ref{err_lem3}.
\end{thm}

\begin{proof}
Given $0 \le n \le N-1$ and $x \in \mathbb{R}$, subtracting \eqref{eq:semi} from \eqref{s3:e3}, we have
\begin{equation}\label{en}
\begin{aligned}
e_n & = \mathbb{E}_{t_n}^x [e_{n+1}] + 
\frac{\Delta t}{2} \mathbb{E}_{t_n}^x [f(t_{n+1}, X_{t_{n+1}}, Y_{t_{n+1}}) - f(t_{n+1}, X_{n+1}, Y_{n+1})]\\
& \hspace{1.5cm} + \frac{\Delta t}{2} [f(t_{n}, X_{t_{n}}, Y_{t_{n}}) - f(t_{n}, X_{n}, Y_{n})] + R_n\\
\end{aligned}
\end{equation}
Due to the Lipschitz continuity of $f$ and the fact that $X_{t_{n+1}} = X_{n+1}$, $X_{t_{n}} = X_{n} = x$, we get the  bound
\begin{equation}
|e_n| \le \mathbb{E}_{t_n}^x [|e_{n+1}|] + \frac{L\Delta t}{2} \mathbb{E}_{t_n}^x[|e_{n+1}|] + \frac{L\Delta t}{2} |e_n| + |R_n|,
\end{equation}
where $L$ is the Lipschitz constant of $f$ and the time step size $\Delta t$ satisfies $1 - \frac{L\Delta t}{2} > 0$. 
Substituting the upper bound of $R_n$, 
the above inequality can be rewritten as
\begin{equation}
|e_n| \le \frac{1+\frac{L\Delta t}{2}}{1- \frac{L\Delta t}{2}} \mathbb{E}_{t_n}^x [|e_{n+1}|] + \frac{C(\Delta t)^3}{1-\frac{L\Delta t}{2}}.
\end{equation}
Taking mathematical expectation on both sides, we obtain, by induction, an upper bound of $\mathbb{E}[|e_n|]$, i.e.,
\begin{equation}
\begin{aligned}
\mathbb{E}[|e_n|] & \le (1+L\Delta t)^{N-n} \mathbb{E}[|e_N|] + C(\Delta t)^3 \frac{(1+L\Delta t)^{N-n}-1}{L\Delta t}\\
& \le \frac{C(\Delta t)^2}{L} (e^{LT} - 1),
\end{aligned}
\end{equation}
where the constant $C$ is the same constant from \eqref{eq:CS1}.
\end{proof}

Next, we focus on analyzing the numerical error for the fully-discrete scheme \eqref{eq:full}. For $0\le n \le N-1$,
$Y_{n,p}$ is constructed using piecewise $p$-th order Lagrange polynomial interpolation on the mesh $\mathcal{S}_h$.
As for the quadrature rules involved in $\widehat{\mathbb{E}}_{t_n,M_y}^{x_i}[\,\cdot\,]$, without loss of generality, 
we consider a special case stated in the following assumption. 
\begin{assum}\label{assump1}
For $d = 1$ and $m = 1, \ldots, M_y$, the quadrature rules $\{w_q^m\}_{q=1}^{Q_m}$, $\{a_q^{1,m}, \ldots,a_q^{m,m} \}_{q=1}^{Q_m}$ 
used in \eqref{appro_EY} are assumed to be the tensor-product of the selected one-dimensional 
quadrature rule, denoted by $\{\widetilde{w}_j, \widetilde{a}_j\}_{j=1}^Q$. The error in the 
quadrature rule is represented by
\begin{equation}
\begin{aligned}
& \Bigg|\int_{E} \cdots \int_{E} Y_{t_{n+1}}\bigg(x_{i} + \sum_{k=1}^m e_k\bigg)
\bigg(\prod_{k=1}^m \rho(e_k) \bigg)de_1\cdots de_m \\
& \hspace{1cm} - \sum_{q_1=1}^Q \cdots \sum_{q_m=1}^Q
\bigg(\prod_{k=1}^m \widetilde{w}_{q_k} \bigg) \cdot Y_{t_{n+1}}\bigg( x_i + \sum_{k=1}^m \widetilde{a}_{q_k}\bigg)\Bigg| \le C Q^{-r},
\end{aligned}
\end{equation}
where the convergence rate $r > 0$ and the constant $C$ are determined by the smoothness of $Y_{n+1}$ and $\rho$. 
The same strategy is also applied to $\widehat{\mathbb{E}}_{t_n,M_f}^{x_i}[f]$.
\end{assum}

Then, with Assumption \ref{assump1}, 
the error estimate for the fully-discrete approximation \eqref{eq:full} is given in the following theorem.
\begin{thm}\label{err_full}
Let $Y_{t_n}$ and ${Y_{n,p}^i}$, for $n = 0, 1, \ldots, N$, $i \in \mathbb{Z}$, be the exact solution of the BSDE \eqref{eq:BSDE2}
and the fully-discrete approximation obtained by the scheme \eqref{eq:full} with $\theta = \frac{1}{2}$, respectively.
Under Assumption \ref{assump1} and the conditions in Theorems \ref{err_lem3} and \ref{TL}, the error 
$e_n^i = Y_{t_n}^{x_i} - Y_{n,p}^i$ can be bounded by
\begin{equation}
\max_{i\in \mathbb{Z}} |e_n^i| \le C \left[ (\Delta t)^2 + 
(\lambda \Delta t)^{M_y} +(\lambda \Delta t)^{M_f+1} + Q^{-r} + (\Delta x)^{p+1}\right],
\end{equation}
where the constant $C$ 
is the same constant from \eqref{eq:CS1}.
\end{thm}

\begin{proof} 
At each grid point $x_i \in \mathcal{S}_h$, substituting $X_{t_n} = x_i$ into \eqref{s3:e3}, we have
\begin{equation}\label{s4:e3}
\begin{aligned}
Y_{t_n}^{x_i} =\mathbb{E}_{t_n}^{x_i}[Y_{t_{n+1}}]
& +\frac{\Delta t}{2} \mathbb{E}_{t_n}^{x_i}[f(t_{n+1},X_{t_{n+1}}, Y_{t_{n+1}})] +\frac{\Delta t}{2} f(t_n, x_i,Y_{t_n}^{x_i}) +R_n,
\end{aligned}
\end{equation}
where $Y_{t_n}^{x_i}$ is the value of the exact solution at the grid point $(t_n, x_i)$. 
Subtracting \eqref{s4:e3} from \eqref{eq:full} leads to
\begin{equation}\label{eni}
\begin{aligned}
e_n^i & = \mathbb{E}_{t_n}^{x_i}[Y_{t_{n+1}}] - \widehat{\mathbb{E}}_{t_n,M_y}^{x_i}[{Y}_{n+1,p}] + \frac{\Delta t}{2} 
 \Big\{\mathbb{E}_{t_n}^{x_i}[f(t_{n+1},X_{t_{n+1}}, Y_{t_{n+1}})]\\
 & \hspace{0.3cm} -  \widehat{\mathbb{E}}_{t_n,M_f}^{x_i}[f(t_{n+1}, X_{n+1}, {Y}_{n+1,p})]\Big\}
 + \frac{\Delta t}{2} [f(t_n, x_i,Y_{t_n}^{x_i}) - f(t_n, x_i, Y_{n,p}^i)] + R_n\\
 & = I_1 + I_2 + I_3 + R_n,
\end{aligned}
\end{equation}
where 
\begin{equation}
\begin{aligned}
& I_1 =  \mathbb{E}_{t_n}^{x_i}[Y_{t_{n+1}}] - \widehat{\mathbb{E}}_{t_n,M_y}^{x_i}[{Y}_{n+1,p}],\\
& I_2 = \frac{\Delta t}{2} \Big\{\mathbb{E}_{t_n}^{x_i}[f(t_{n+1},X_{t_{n+1}}, Y_{t_{n+1}})]
-  \widehat{\mathbb{E}}_{t_n,M_f}^{x_i}[f(t_{n+1}, X_{n+1}, {Y}_{n+1,p})]\Big\},\\
& I_3 = \frac{\Delta t}{2} [f(t_n, x_i,Y_{t_n}^{x_i}) - f(t_n, x_i, {Y}_{n,p}^i)].
\end{aligned}
\end{equation}
Due to the Lipschitz continuity of $f$, it is easy to see that
\begin{equation}\label{I3}
|I_3| \le C \Delta t |e_n^i|,
\end{equation}
where $C$ depends on the Lipschitz constant of $f$. We can split $I_1$ in the following way
\begin{equation}
\begin{aligned}
I_1 & = \underbrace{\mathbb{E}_{t_n}^{x_i}[Y_{t_{n+1}}] - {\mathbb{E}}_{t_n,M_y}^{x_i}[Y_{t_{n+1}}]}_{I_{11}} 
+  \underbrace{{\mathbb{E}}_{t_n,M_y}^{x_i}[Y_{t_{n+1}}] - \widehat{\mathbb{E}}_{t_n,M_y}^{x_i}[Y_{t_{n+1}}]}_{I_{12}} \\
& \hspace{0.4cm}+ \underbrace{\widehat{\mathbb{E}}_{t_n,M_y}^{x_i}[Y_{t_{n+1}} - {Y}_{t_{n+1},p}]}_{I_{13}}+
 \underbrace{\widehat{\mathbb{E}}_{t_n,M_y}^{x_i}[{Y}_{t_{n+1},p}-{Y}_{n+1,p}]}_{I_{14}},
\end{aligned}
\end{equation}
where $Y_{t_{n+1},p}$ is the $p$-th order Lagrange polynomial interpolant of $Y_{t_{n+1}}$. 
By the definition of $\mathbb{E}_{t_n}^{x_i}[\,\cdot\,]$ and ${\mathbb{E}}_{t_n,M_y}^{x_i}[\,\cdot\,]$ in \eqref{eq:EY} and \eqref{appro_EY},
we have
\begin{equation}
\begin{aligned}
|I_{11}| & = \sum_{m = M_y+1}^\infty e^{-\lambda\Delta t} \frac{(\lambda \Delta t)^m }{m!} 
\int_{E} \cdots \int_{E}  Y_{t_{n+1}}\bigg(x_i + \sum_{k=1}^m e_k\bigg)
\bigg(\prod_{k=1}^m \rho(e_k) \bigg)de_1\cdots de_m\\
& \le C (\lambda \Delta t)^{M_y+1},
\end{aligned}
\end{equation}
for $\lambda \Delta t < 1$, where $C$ depends on the upper bounds of $|Y_{t_{n+1}}|$ and $|\rho|$. For the error $I_{12}$, 
we obtain
\begin{equation}
\begin{aligned}
 |I_{12}| = & \left|{\mathbb{E}}_{t_n,M_y}^{x_i}[Y_{t_{n+1}}] - \widehat{\mathbb{E}}_{t_n,M_y}^{x_i}[Y_{t_{n+1}}] \right|\\
 = & \Bigg| \sum_{m=M_y+1}^{\infty} e^{-\lambda\Delta t} \frac{(\lambda \Delta t)^m }{m!} 
 \Bigg\{\int_{E} \cdots \int_{E}  Y_{t_{n+1}}\bigg(x_i + \sum_{k=1}^m e_k\bigg)
\bigg(\prod_{k=1}^m \rho(e_k) \bigg)de_1\cdots de_m\\
&\hspace{2cm} -  \sum_{q_1=1}^Q \cdots \sum_{q_m=1}^Q
\bigg(\prod_{k=1}^m \widetilde{w}_{q_k} \bigg) \cdot Y_{t_{n+1}}\bigg( x_i + \sum_{k=1}^m \widetilde{a}_{q_k}\bigg)\Bigg\} \Bigg|\\
\le & \sum_{m=M_y+1}^{\infty} e^{-\lambda\Delta t} \frac{(\lambda \Delta t)^m }{m!} \cdot \left( C Q^{-r}  \right)\\
\le &  C  \lambda \Delta t Q^{-r}
\end{aligned}
\end{equation}
for $\lambda \Delta t < 1$ where $C$ depends on the upper bound of the derivatives of $Y_{t_{n+1}}$. 
For $I_{13}$, based on the fact that $Y_{t_{n+1}}(x_i) - {Y}_{t_{n+1},p}(x_i) = 0$ for $x_i \in \mathcal{S}_h$ and 
the classic error bound of Lagrange interpolation, we have
\begin{equation}
\begin{aligned}
 |I_{13}| & = \left|\widehat{\mathbb{E}}_{t_n,M_y}^{x_i}\left[Y_{t_{n+1}} - {Y}_{t_{n+1},p}\right] \right|\\
 & \le e^{-\lambda\Delta t} \left|Y_{t_{n+1}}(x_i) - {Y}_{t_{n+1},p}(x_i)\right|
 +  \sum_{m=1}^{M_y} e^{-\lambda\Delta t} \frac{(\lambda \Delta t)^m }{m!}  \\
& \hspace{0.5cm} \cdot \Bigg[ \sum_{q_1=1}^Q \cdots \sum_{q_m=1}^Q
\bigg(\prod_{k=1}^m \widetilde{w}_{q_k} \bigg) \left|Y_{t_{n+1}}\bigg( x_i + \sum_{k=1}^m \widetilde{a}_{q_k}\bigg) - 
{Y}_{t_{n+1},p}\bigg( x_i + \sum_{k=1}^m \widetilde{a}_{q_k}\bigg)\right|\Bigg]\\
& \le C (\Delta x)^{p+1} \cdot \sum_{m=1}^{M_1} e^{-\lambda\Delta t} \frac{(\lambda \Delta t)^m }{m!}\\
& \le C \lambda \Delta t (\Delta x)^{p+1}.
\end{aligned}
\end{equation}
Also, for $I_{14}$, we deduce that
\begin{equation}
\begin{aligned}
 |I_{14}| & \le \left|\widehat{\mathbb{E}}_{t_n,M_y}^{x_i}[{Y}_{t_{n+1},p}- {Y}_{n+1,p}]\right| \\
 & \le   e^{-\lambda\Delta t} \left|{Y}_{t_{n+1},p} (x_i) - {Y}_{n+1,p}(x_i) \right| + 
   \sum_{m=1}^{M_y} e^{-\lambda\Delta t} \frac{(\lambda \Delta t)^m }{m!} \cdot \Lambda_p^{\Delta x} \max_{i\in \mathbb{Z}}  |e_{n+1}^i| \\
 & \le \left(1 + C\Delta t \Lambda_p^{\Delta x}\right) \max_{i\in \mathbb{Z}} |e_{n+1}^i|,
\end{aligned}
\end{equation}
where $\Lambda_p^{\Delta x}$ is the Lebesgue constant of the $p$-th order Lagrange interpolant under the mesh $\mathcal{S}_h$.

Combining $I_{11}, I_{12}, I_{13}, I_{14}$, we obtain the following upper bound of $|I_1|$, i.e.,
\begin{equation}\label{I1}
|I_1| \le C\left[ (\lambda \Delta t)^{M_y+1} + \lambda \Delta t Q^{-r} + \lambda \Delta t (\Delta x)^{p+1}\right] + 
\left(1 + C\Delta t \Lambda_p^{\Delta x}\right)\max_{i\in \mathbb{Z}} |e_{n+1}^i|.
\end{equation}
Similarly, we can obtain the following upper bound of $|I_2|$, i.e.,
\begin{equation}\label{I2}
|I_2| \le C\Delta t \left[ (\lambda \Delta t)^{M_f+1} + \lambda \Delta t Q^{-r} + \lambda \Delta t(\Delta x)^{p+1} 
+ \left(1 + C\Delta t \Lambda_p^{\Delta x}\right)\max_{i\in \mathbb{Z}} |e_{n+1}^i|\right].
\end{equation}
Substituting \eqref{I1}, \eqref{I2} and \eqref{I3} into \eqref{eni}, we have
\begin{equation}
\begin{aligned}
&(1 - C\Delta t) \max_{i\in \mathbb{Z}} |e_n^i| \le (1 + C\Delta t) \max_{i\in \mathbb{Z}} |e_{n+1}^i| \\
&\hspace{1cm}+ C\left[(\lambda \Delta t)^{M_y+1} + \Delta t (\lambda \Delta t)^{M_f+1}\right] + \lambda \Delta tQ^{-r} + C\lambda \Delta t(\Delta x)^{p+1} + C(\Delta t)^3.
\end{aligned}
\end{equation}
Therefore, for sufficiently small $\Delta t$, we obtain, by induction, 
\begin{equation}
\max_{i\in \mathbb{Z}} |e_n^i| \le C \left[ (\Delta t)^2 + 
(\lambda \Delta t)^{M_y} +(\lambda \Delta t)^{M_f+1} + Q^{-r} + (\Delta x)^{p+1}\right].
\end{equation}
\end{proof}
%

%
%
\section{Numerical examples}\label{sec:ex}
In this section, we report on the results of three one-dimensional numerical examples, that illustrate the accuracy
and efficiency of the proposed schemes \eqref{eq:full}. We take uniform partitions in both temporal and 
spatial domains with time and space step sizes $\Delta t$ and $\Delta x$, respectively. 
The number of time steps is denoted by $N$, which is given by $N = \frac{T}{\Delta t}$, with $T$ is the terminal time. 
For the sake of illustration, we only solve the nonlocal problems on bounded spatial domains. Lagrange interpolation and
tensor-product quadrature rules are used to approximate $\mathbb{E}_{t_n}^{x_i}[\,\cdot\,]$, where the one-dimensional 
quadrature rule for each example is chosen according to the property of the kernel $\gamma(e)$, the initial condition $u_0$
and the forcing term $f$. Our main goal is to test the accuracy and convergence rate of the proposed scheme \eqref{eq:full}
with respect to $\Delta t$ and $\Delta x$. To this end, according to the analysis in Theorem \ref{err_full}, we always set the number
of quadrature points to be sufficiently large so that the error contributed by the use of quadrature rules is too small to
affect the convergence rate. 

\subsection{Symmetric kernel}\label{sec:ex1}
First we consider the following nonlocal diffusion problem in $[0,T]$:
\begin{equation}\label{ex1:eq}
\left\{
\begin{aligned}
& \frac{\partial u}{\partial t} - \frac{1}{\delta^3} \int_{-\delta}^{\delta} \Big(u(t,x+e)-u(t,x)\Big) de  = g(t,x), \;\; t > 0,\\
& u(0,x)  = \varphi(x),\\
\end{aligned}\right.
\end{equation}
where $\delta > 0$ which corresponds to the symmetric kernel  
\begin{equation}
\gamma(e) = \left\{
\begin{aligned}
& \frac{1}{\delta^3}, & \text{for } e \in [-\delta, \delta],\\
& 0, & \text{for } e \notin [-\delta, \delta].\\
\end{aligned}\right.
\end{equation}
We choose the exact solution to be 
\[
u(t,x) = (-x^3 + x^2) \exp\Big(-\frac{t}{10}\Big),
\]
so that the forcing term $g$ is given by
\[
g(t,x) = -\frac{u(t,x)}{10} + \Big(2x - \frac{2}{3}\Big) \exp\Big(-\frac{t}{10}\Big),
\]
and the initial condition $\varphi(x)$ can be determined from $u$ accordingly. After converting the problem \eqref{ex1:eq} 
to a BSDE of the form in \eqref{eq:BSDE2}, we have
\[
\lambda = \frac{2}{\delta^2}, \;\;  \rho(e) = \frac{1}{2\delta} \mathcal{I}_{[-\delta, \delta]}(e),\;\; b = 0,\;\; f(t, \cdot, \cdot) = g(T-t, \cdot, \cdot),
\]
where $\mathcal{I}_{[-\delta, \delta]}(e)$ is a characteristic function of the interval $[-\delta, \delta]$. 

We set the terminal time $T = 1$ and solve \eqref{ex1:eq} on the spatial domain $[0,1] \in \mathbb{R}$. 
Since the density function $\rho(e)$ is uniform with the support $[-\delta, \delta]$, we use tensor-product 
Gauss-Legendre quadrature rules, with $Q = 16$, to approximate the integrals involved in ${\mathbb{E}}_{t_n}^{x_i}[\,\cdot\,]$.
First, we test the convergence rate with respect to $\Delta t$. To this end, we set $N_x = 65$ and use piecewise cubic Lagrange 
interpolation ($p=3$) to construct $Y_{n,p}$ for $n = 0, \ldots, N-1$, so that the time discretization error dominates the total error.
For $\delta = 1$, the number $N$ of time steps is set to $4, 8, 16, 32, 64$, respectively. The numerical results are shown in Table \ref{ex1:t1}.
As expected, the convergence rate with respect to $\Delta t$ depends not only on the value of $\theta$, but also the number of jumps
retained in constructing $\mathbb{E}_{t_n,M_y}^{x_i}[\,\cdot\,]$ and $\mathbb{E}_{t_n,M_f}^{x_i}[\,\cdot\,]$. For example, when $M_y = M_f = 0$, i.e., 
no jump is included in $\mathbb{E}_{t_n,M_y}^{x_i}[\,\cdot\,]$ and $\mathbb{E}_{t_n,M_f}^{x_i}[\,\cdot\,]$, the scheme \eqref{eq:full} fails to converge.
In general, in order to achieve a $k$-th order convergence ($k =1,2$), $M_y$ and $M_f$ must satisfy $M_y \ge k$ and $M_f \ge k-1$. 
\begin{table}[h!]
 \footnotesize
\begin{center}
\caption{Errors and convergence rates with respect to $\Delta t$ in Example \ref{sec:ex1} 
where $T =1$, $\delta = 1$, $N_x = 65$, $p=3$. } \label{ex1:t1}
\begin{tabular}{|c|c|c|c|c|c|c|}\hline
\multicolumn{7}{|c|}{$\| Y_{t_0} -Y_{0,p}\|_{\infty}$}\\
\hline   & $N=4$ & $N=8$ & $N=16$ & $N=32$ & $N=64$ & CR \\
\hline
$\theta = 0$, $M_y = 0$, $M_f=0$   &    1.932E-1  &   1.978E-1  &   2.000E-1  &   2.012E-1  &   2.017E-1  &   -0.015  \\ 
\hline
 $\theta = 0$, $M_y = 1$, $M_f=0$ &   7.958E-2  &   3.435E-2  &   1.572E-2  &   7.487E-3  &   3.649E-3  &   1.109  \\
 \hline
 $\theta = 0$, $M_y = 2$, $M_f=1$ &   3.673E-2  &   1.849E-2  &   9.279E-3  &   4.675E-3  &   2.326E-3  &   0.995  \\
 \hline\hline
 $\theta=1$, $M_y = 0$, $M_f=0$  &    2.111E-1  &   2.067E-1  &   2.045E-1  &   2.034E-1  &   2.028E-1  &   0.014  \\
  \hline
 $\theta=1$, $M_y = 1$, $M_f=0$  &    4.891E-2  &   2.581E-2  &   1.328E-2  &   6.736E-3  &   3.393E-3  &   0.964  \\
 \hline
 $\theta=1$, $M_y = 2$, $M_f=1$  &    6.170E-2  &   3.007E-2  &   1.468E-2  &   7.231E-3  &   3.585E-3  &   1.027  \\
 \hline\hline
 $\theta=\frac{1}{2}$, $M_y = 0$, $M_f=0$ &  2.030E-1  &   2.025E-1  &   2.023E-1  &   2.023E-1  &   2.023E-1  &   0.001  \\
  \hline
 $\theta=\frac{1}{2}$, $M_y = 1$, $M_f=0$  &    2.623E-2  &   1.326E-2  &   6.666E-3  &   3.342E-3  &   1.674E-3  &   0.993  \\
 \hline
 $\theta=\frac{1}{2}$, $M_y = 2$, $M_f=1$ &  3.207E-3  &   8.258E-4  &   2.094E-4  &   5.271E-5  &   1.322E-5  &   1.981  \\
\hline
 $\theta=\frac{1}{2}$, $M_y = 3$, $M_f=2$ &  3.330E-3  &   8.632E-4  &   2.196E-4  &   5.538E-5  &   1.391E-5  &   1.977  \\
\hline
\end{tabular}
\end{center}
\end{table}

We observe in Theorem \ref{err_full} that the errors with respect to $\Delta t$ from $\mathbb{E}_{t_n,M_y}^{x_i}[\,\cdot\,]$ 
and $\mathbb{E}_{t_n,M_f}^{x_i}[\,\cdot\,]$ are of order $(\lambda \Delta t)^{M_y}$ and $(\lambda \Delta t)^{M_f+1}$, respectively.
If the intensity $\lambda$ is large, i.e., the horizon $\delta$ is small, the value of $\delta$ will affect the error decay 
in the pre-asymptotic region and change the total error up to a constant. Such phenomenon is investigated by setting $\delta = 0.3$ and $0.4$, 
and $\theta = 1$ and $\frac{1}{2}$. The results are shown in Figure \ref{ex1:f1}. In Figure \ref{ex1:f1}(a) with $\theta = 1$, 
in the case that $\delta = 0.3$ ($\lambda \approx  22$), $M_y = 1$ and $M_f = 0$, 
the error is actually of order $\mathcal{O} (\lambda \Delta t)$. When one more jump is included, i.e., $M_y = 2$ and $M_f = 1$, we observe that
the convergence rate remains the same but the total error is significantly reduced, since the error contributed by
$\mathbb{E}_{t_n,M_y}^{x_i}[\,\cdot\,]$ and $\mathbb{E}_{t_n,M_f}^{x_i}[\,\cdot\,]$ is reduced to $\mathcal{O}((\lambda \Delta t)^2)$.
\begin{figure}[h!]
\begin{center}
\includegraphics[scale =0.4]{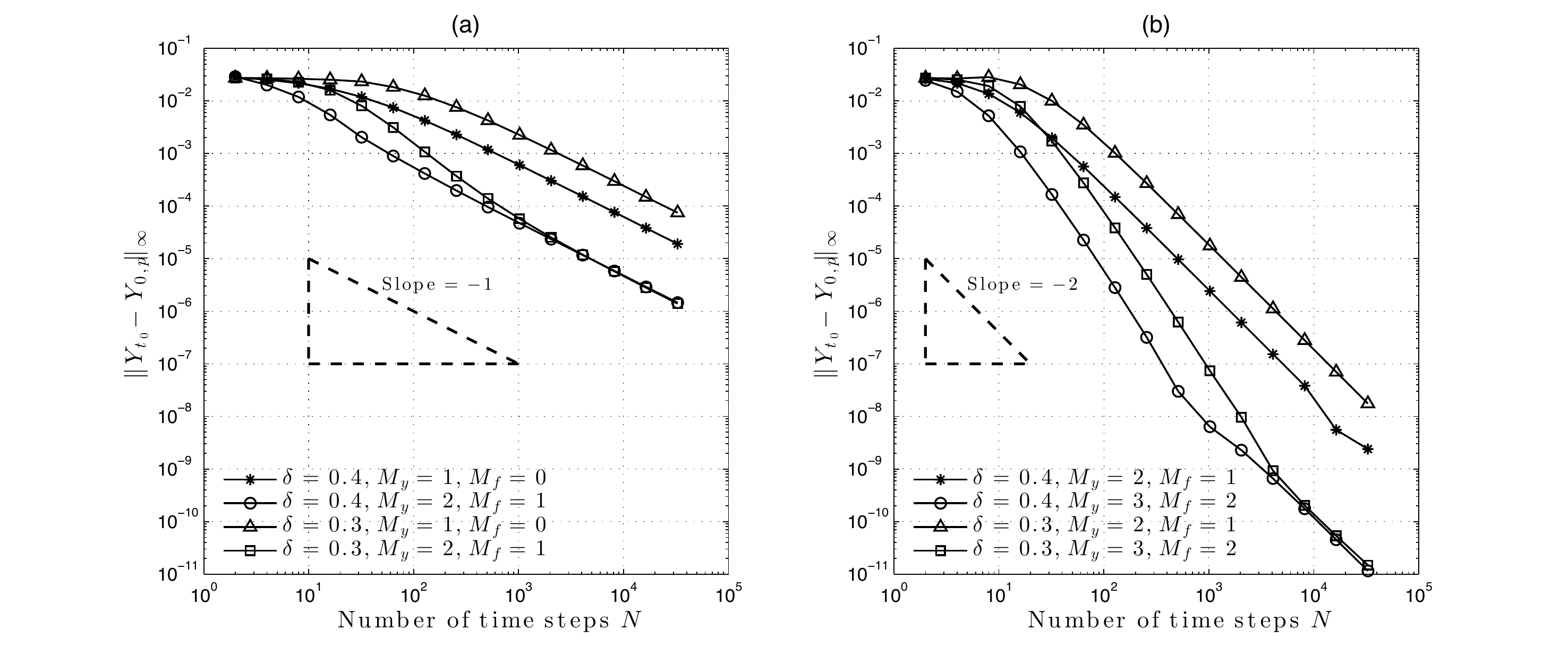}
\end{center}\caption{Error decays for (a) $\theta = 1$ and (b) $\theta = \frac{1}{2}$ in Example \ref{sec:ex1}.}\label{ex1:f1}
\end{figure}

Next, we test the convergence rate with respect to $\Delta x$ by setting $\delta = 1$, $T = 1$, $\theta = \frac{1}{2}$, $N = 1024$,
$M_y = 3$ and $M_f = 2$.
We divide the spatial interval $[0,1]$ into $N_x$ elements with $\Delta x = 2^{-3}, 2^{-4}, 2^{-5}, 2^{-6}, 2^{-7}$.
The error is measured in $L^\infty$ and $L^2$ norms. In Table \ref{ex1:t2}, we can see that the numerical results verify
the theoretical analysis in \S \ref{sec:err}.
\begin{table}[h!]
 \footnotesize
\begin{center}
\caption{Errors and convergence rates with respect to $\Delta x$ in Example
\ref{sec:ex1} where $\delta = 1$, $T = 1$, $\theta = \frac{1}{2}$, $N = 1024$, $M_y = 3$ and $M_f = 2$.} \label{ex1:t2}
\begin{tabular}{|c||c|c||c|c|}\hline
 & \multicolumn{2}{|c||}{Linear interpolation} &  \multicolumn{2}{|c|}{Quadratic interpolation}\\
\hline   
 $\Delta x$ & $\|Y_{t_0}-Y_{0,1} \|_{L^2}$ &  $\|Y_{t_0}-Y_{0,1} \|_{L^\infty}$ &  $\|Y_{t_0}-Y_{0,2} \|_{L^2}$ &  $\|Y_{t_0} -Y_{0,2}\|_{L^\infty}$\\
\hline
$2^{-3}$ & 1.227E-03 &  2.816E-03 & 1.862E-04  &  2.969E-04 \\
\hline
$2^{-4}$ & 3.266E-04 &  7.356E-04 & 2.278E-05  &  3.881E-05 \\
\hline
$2^{-5}$ & 6.586E-05 &  1.772E-04 & 2.792E-06  &  4.661E-06 \\
\hline
$2^{-6}$ & 1.651E-05 &  4.481E-05 & 3.491E-07  & 6.252E-07 \\
\hline
$2^{-7}$ & 4.167E-06 &  1.103E-05 & 4.738E-08  & 9.188E-08 \\
\hline
CR         & 2.071        &   2.001        &     2.991      &    2.927   \\ 
\hline
\end{tabular}
\end{center}
\end{table}

\subsection{Singular kernel}\label{sec:ex3}
We consider the following nonlocal diffusion problem in $[0,T]$
\begin{equation}\label{ex3:eq}
\left\{
\begin{aligned}
& \frac{\partial u}{\partial t} - \frac{1}{\delta^2\sqrt{\delta}} \int_{-\delta}^{\delta} \frac{u(t,x+e)-u(t,x)}{\sqrt{|e|}}de  = g(t,x), \;\; t > 0,\\
& u(0,x)  = \varphi(x),\\
\end{aligned}\right.
\end{equation}
where $\delta > 0$ which corresponds to the singular kernel $\gamma(e)$  
\begin{equation}
\gamma(e) = \left\{
\begin{aligned}
& \frac{1}{\delta^2\sqrt{\delta |e|}} , & \text{for } e \in [-\delta, \delta],\\
& 0, & \text{for } e \notin [-\delta, \delta].\\
\end{aligned}\right.
\end{equation}
We choose the exact solution to be 
\[
u(t,x) = (x+t)^2,
\]
so that the forcing term $g$ is given by
\[
g(t,x) =  2(x+t) - \frac{4}{5}
\]
and the initial condition $\varphi(x)$ can be determined from $u$ accordingly. After converting the problem \eqref{ex1:eq} 
to a BSDE of the form in \eqref{eq:BSDE2}, we have
\[
\lambda = \frac{4}{\delta^2}, \;\; \rho(e) = \frac{1}{4\sqrt{\delta|e|}}\mathcal{I}_{[-\delta, \delta]}(e),\;\; b = 0,\;\; f(t, \cdot, \cdot) = g(T-t, \cdot, \cdot),
\]
where $\mathcal{I}_{[-\delta, \delta]}(e)$ is a characteristic function of the interval $[-\delta, \delta]$. 

We set the terminal time $T = 0.25$ and solve the equation \eqref{ex3:eq} on the spatial domain $[0,1] \in \mathbb{R}$.
It is observed that the density function $\rho(e)$ is singular at $e = 0$, but it is still integrable in the interval $[-\delta ,\delta]$.
Recalling that ${1}/{\sqrt{|e|}}$ is a special case of the kernel of the Gauss-Jacobi quadrature rule, i.e.
$\int_{a}^b \psi(x) (b-x)^{\alpha} (x-a)^\beta dx$ with $\alpha, \beta \in [-1,1]$, where the kernel is given by $(b-x)^{\alpha} (x-a)^\beta$. 
Thus, the integrals involved in 
$\mathbb{E}_{t_n,M_y}^{x_i}[\,\cdot\,]$ and $\mathbb{E}_{t_n,M_f}^{x_i}[\,\cdot\,]$ can be accurately approximated by setting
$a = 0$, $b = \delta$, $\alpha = 0$ and $\beta = -\frac{1}{2}$. In this example, we use 16-point Gauss-Jacobi rule such that
the quadrature error can be ignored. We test the convergence rates with respect to $\Delta t$ by setting $\delta = 1$, $N_x = 65$ and 
$p = 3$, and the results are given in Table \ref{ex3:t1}. As expected, due to the use of Gauss-Jacobi rule, the temporal truncation error 
dominates the total error and the theoretical result in Theorem \ref{err_full} has been verified as well. The convergence with
respect to $\Delta x$ is also tested by setting $p = 1$, $T = 0.25$, $\theta = \frac{1}{2}$, $N = 1024$, $M_y = 3$ and $M_f = 2$.
Table \ref{ex2:t2} shows the results in two cases where $\delta = 1$ and 0.1, respectively. We see that, for different horizon values, the spatial discretization error decays at the same rate verifying the theoretical error estimates in \S \ref{sec:err}.
\begin{table}[h!]
 \footnotesize
\begin{center}
\caption{Errors and convergence rates with respect to $\Delta t$ in Example
\ref{sec:ex3} where $T =0.25$, $\delta = 1$, $N_x = 65$, $p=3$.  } \label{ex3:t1}
\begin{tabular}{|c|c|c|c|c|c|c|}\hline
\multicolumn{7}{|c|}{$\| Y_{t_0} -Y_{0,p}\|_{\infty}$}\\
\hline   & $N=4$ & $N=8$ & $N=16$ & $N=32$ & $N=64$ & CR \\
\hline
$\theta = 0$, $M_y=0, M_f=0$   &    9.732E-1  &   9.554E-1  &   9.461E-1  &   9.413E-1  &   9.390E-1  &   0.013 \\ 
\hline
 $\theta = 0$, $M_y = 1, M_f = 0$ &   2.099E-1  &   1.108E-1  &   5.698E-2  &   2.890E-2  &   1.456E-2  &   0.964  \\
 \hline
 $\theta = 0$, $M_y = 2, M_f = 1$ &   8.618E-2  &   4.131E-2  &   2.013E-2  &   9.925E-3  &   4.926E-3  &   1.032  \\
 \hline\hline
 $\theta=1$, $M_y = 0,M_f =0$  &    8.958E-1  &   9.167E-1  &   9.267E-1  &   9.317E-1  &   9.341E-1  &   -0.014  \\
  \hline
 $\theta=1$, $M_y = 1,M_f =0$  &    1.239E-1  &   6.669E-2  &   3.461E-2  &   1.763E-2  &   8.900E-3  &   0.952  \\
 \hline
 $\theta=1$, $M_y = 2,M_f =1$  &    8.485E-2  &   5.274E-2  &   2.959E-2  &   1.577E-2  &   8.149E-3  &   0.950  \\
 \hline\hline
 $\theta=\frac{1}{2}$, $M_y = 0, M_f =0$ &  9.345E-1  &   9.360E-1  &   9.364E-1  &   9.365E-1  &   9.365E-1  &   -0.001  \\
  \hline
 $\theta=\frac{1}{2}$, $M_y = 1, M_f =0$  &    1.669E-1  &   8.874E-2  &   4.579E-2  &   2.327E-2  &   1.173E-2  &   0.959  \\
 \hline
 $\theta=\frac{1}{2}$, $M_y = 2, M_f =1$ &  1.634E-2  &   4.386E-3  &   1.136E-3  &   2.889E-4  &   7.286E-5  &   1.954  \\
\hline
 $\theta=\frac{1}{2}$, $M_y = 3, M_f =2$ &  4.476E-3  &   1.079E-3  &   2.634E-4  &   6.498E-5  &   1.613E-5  &   2.029  \\
\hline
\end{tabular}
\end{center}
\end{table}

\begin{table}[h!]
 \footnotesize
\begin{center}
\caption{Errors and convergence rates with respect to $\Delta x$ in Example
\ref{sec:ex3} where $p = 1$, $T=0.25$, $\theta = 1/2$, $N = 1024$, $M_y = 3$ and $M_f = 2$.} \label{ex2:t2}
\begin{tabular}{|c||c|c||c|c|}\hline
 & \multicolumn{2}{|c||}{$\delta  = 1$} &  \multicolumn{2}{|c|}{$\delta  = 0.1$}\\
\hline   
 $\Delta x$ & $\|Y_{t_0} - Y_{0,p} \|_{L^2}$ &  $\|Y_{t_0} - Y_{0,p} \|_{L^\infty}$ &  $\|Y_{t_0} - Y_{0,p} \|_{L^2}$ &  $\|Y_{t_0} - Y_{0,p} \|_{L^\infty}$\\
\hline
$2^{-3}$ & 3.912E-03 &  5.232E-03 & 7.892E-03  &  1.046E-02 \\
\hline
$2^{-4}$ & 9.934E-04 &  1.318E-03 & 2.030E-03  &  2.618E-03 \\
\hline
$2^{-5}$ & 2.647E-04 &  3.482E-04 & 4.893E-04  &  5.787E-04 \\
\hline
$2^{-6}$ & 6.062E-05 &  8.057E-05 & 1.269E-04  &  1.180E-04 \\
\hline
$2^{-7}$ & 1.318E-05 &  1.862E-05 & 1.498E-05  & 2.541E-05 \\
\hline
CR         & 2.046        &   2.030        &     2.208      &    2.184   \\ 
\hline
\end{tabular}
\end{center}
\end{table}

\subsection{Non-symmetric kernel and discontinuous solution}\label{sec:ex2}
We consider the following nonlocal diffusion problem in $[0,T]$,
\begin{equation}\label{ex2:eq}
\left\{
\begin{aligned}
& \frac{\partial u}{\partial t} - \int_{-\delta}^{2\delta} \Big[u(t,x+e)-u(t,x)\Big] de  = g(t,x), \;\; t > 0,\\
& u(0,x)  = \varphi(x),\\
\end{aligned}\right.
\end{equation}
where $\delta > 0$ for which we have the non-symmetric kernel 
\begin{equation}
\gamma(e) = \left\{
\begin{aligned}
& 1, & \text{ if } e \in [-\delta, 2\delta],\\
& 0, & \text{ if } e \notin [-\delta, 2\delta].\\
\end{aligned}\right.
\end{equation}
We choose the exact solution 
\begin{equation}
u(t,x) = \left\{
\begin{aligned}
& x \sin(t), \hspace{0.3cm}   \text{ if } \; x < \frac{1}{2},\\
& x^2  \sin(t),  \hspace{0.15cm}\text{ if } \; x \ge \frac{1}{2}\\
\end{aligned}\right.
\end{equation}
so that the forcing term $g$ is given by
\begin{equation}
g(t,x) = \left\{
\begin{aligned}
& \sin(t) \left[ -\frac{(x+2\delta)^2}{2} + \frac{(x-\delta)^2}{2} + 3\delta x\right]  + x\cos(t), \hspace{1.5cm}  x <\frac{1}{2}-2\delta,\\
& \sin(t) \Bigg[-  \frac{1}{12} - \frac{(x+2\delta)^3}{3} + \frac{(x-\delta)^2}{2} + 3\delta x \Bigg] + x\cos(t),\;\;\;\; \frac{1}{2}-2\delta \le x< \frac{1}{2},\\
& \sin(t) \Bigg[-  \frac{1}{12} - \frac{(x+2\delta)^3}{3} + \frac{(x-\delta)^2}{2} + 3\delta x^2 \Bigg] + x^2\cos(t),\; \frac{1}{2}\le x <\frac{1}{2}+\delta,\\
& \sin(t) \Bigg[ - \frac{(x+2\delta)^3}{3} + \frac{(x-\delta)^3}{3} + 3\delta x^2\Bigg] + x^2\cos(t),\hspace{1.0cm} x\ge \frac{1}{2}+\delta\\
\end{aligned}
\right.
\end{equation}
and the initial condition $\varphi(x)$ can be determined from $u$ accordingly. After converting the problem \eqref{ex2:eq} 
to a BSDE of the form in \eqref{eq:BSDE2}, we have
\[
\lambda = 3\delta, \;\; \rho(e) = \frac{1}{3\delta}\mathcal{I}_{[-\delta,2\delta]}(e),\quad b = \frac{3}{2} \delta^2,\quad f(t, \cdot, \cdot) = g(T-t, \cdot, \cdot).
\]

Note that the kernel $\gamma(e)$ is non-symmetric so the drift coefficient $b$ is non-zero. However, as shown in \eqref{b},
non-zero $b$ is equal to the compensator of the underlying compound Poisson process, such that it cancels out the compensator 
in the forward SDE in \eqref{eq:BSDE2}, so that it does not appear in the numerical scheme \eqref{eq:full}. The most challenging aspect of this problem 
is the discontinuities in $u$ and $g$, which deteriorate the accuracy of the quadrature rule and the polynomial interpolation. 
Here, the trapezoidal quadrature rule so that the integral of the interpolating 
polynomials are evaluated exactly. 
To study the convergence of the scheme \eqref{eq:full} with respect to $\Delta x$, we set $T=0.5$, $\theta = \frac{1}{2}$ and $N = 512$.
First, we solve \eqref{ex2:eq} on a uniform spatial grid with $\Delta x = 2^{-3}, 2^{-4},2^{-5},2^{-6},2^{-7}$. The errors and convergence
rates are given in Table \ref{ex4:t1}. As we expected, due to the discontinuity, the $L^{\infty}$ error does not converge and the $L^2$ error converges
with $\mathcal{O}((\Delta x)^{\frac{1}{2}})$. 
\begin{table}[h!]
 \footnotesize
\begin{center}
\caption{Errors and convergence rates with respect to $\Delta x$ in Example 
\ref{sec:ex2} for $p = 1$, $T=0.5$, $\theta = \frac{1}{2}$, $N = 512$.} \label{ex4:t1}
\begin{tabular}{|c||c|c||c|c|}\hline
 & \multicolumn{2}{|c||}{$\delta  = 1$} &  \multicolumn{2}{|c|}{$\delta  = 0.1$}\\
\hline   
 $\Delta x$ & $\|Y_{t_0} - Y_{0,p} \|_{L^2}$ &  $\|Y_{t_0} - Y_{0,p} \|_{L^\infty}$ &  $\|Y_{t_0} - Y_{0,p} \|_{L^2}$ &  $\|Y_{t_0} - Y_{0,p} \|_{L^\infty}$\\
\hline
$2^{-3}$ & 3.001E-02 &  1.323E-01 & 2.503E-02  &  1.215E-01 \\
\hline
$2^{-4}$ & 2.287E-02 &  1.317E-01 & 1.751E-02  &  1.207E-01 \\
\hline
$2^{-5}$ & 1.467E-02 &  1.257E-01 & 1.231E-02  &  1.201E-01 \\
\hline
$2^{-6}$ & 1.107E-02 &  1.254E-01 & 8.676E-03  &  1.197E-01 \\
\hline
$2^{-7}$ & 7.045E-03 &  1.221E-01 & 6.126E-03  &  1.192E-01 \\
\hline
CR         & 0.523        &   0.030        &     0.507      &    0.007   \\ 
\hline
\end{tabular}
\end{center}
\end{table}

In order to improve the convergence rate with respect to the $L^2$ norm, one strategy is to utilize adaptive grids which are capable of automatically refining the spatial grid 
around the discontinuity. Here, we apply the one-dimensional adaptive hierarchical finite element method \cite{Gunzburger:2014hi,Zhang:2013en}. 
This approach can be seamlessly integrated into the sparse grid framework to alleviate the curse of dimensionality when solving high-dimensional 
nonlocal diffusion problems. Figure \ref{ex3:f1}(a) shows the solution surface $u(t,x)$ we recovered in order to approximate $u(0.5,x)$ within the spatial domain $[0,1]$; and Figure \ref{ex3:f1}(b) illustrates
the exact solution $u(0.5,x)$ and its approximation using 33 spatial grid points for which the $L^2$ error is only 9.745E-04. Further illustration is shown in Figure 
\ref{ex3:f2}(a) and Figure 
\ref{ex3:f2}(b) where, for the scheme \eqref{eq:full}, the error decay vs. the number of interpolating points and vs. the grid size $\Delta x$, respectively, of using uniform and adaptive grids are plotted. For the adaptive case, $\Delta x$ is the starting grid for the refinement process. We clearly see the half-order and the optimal second-order convergence rates for the uniform and adaptive grid cases, respectively.
We also study the 
convergence with respect to $\Delta t$ for $\theta = 1$ with the error tolerance for the adaptive grid is set to 0.01, 0.005, 0.001 and 0.0005. 
In Figure \ref{ex3:f2}(c), it can be seen that the desired convergence rate with respect to $\Delta t$ is achieved 
if the error is smaller than the tolerance; otherwise, the spatial discretization error will dominate and the total error will not decay as $\Delta t$ decreases.

Of course, no amount of refinement, whether adaptive or not, will reduce the $L^\infty$ norm error, i.e., it will remain ${\mathcal O}(1)$. However, if we omit the single element containing the discontinuity, we see the following from Figure \ref{ex3:f2}(d). For uniform grid refinement, the error remains ${\mathcal O}(1)$ due to pollution of the error into neighboring elements which are all of size $\Delta x$. However, for adaptive grid refinement, the $L^\infty$ error is close to ${\mathcal O}(h^2)$ because pollution is to neighboring elements of very small size.

\begin{figure}[h!]
\begin{center}
\includegraphics[scale =0.33]{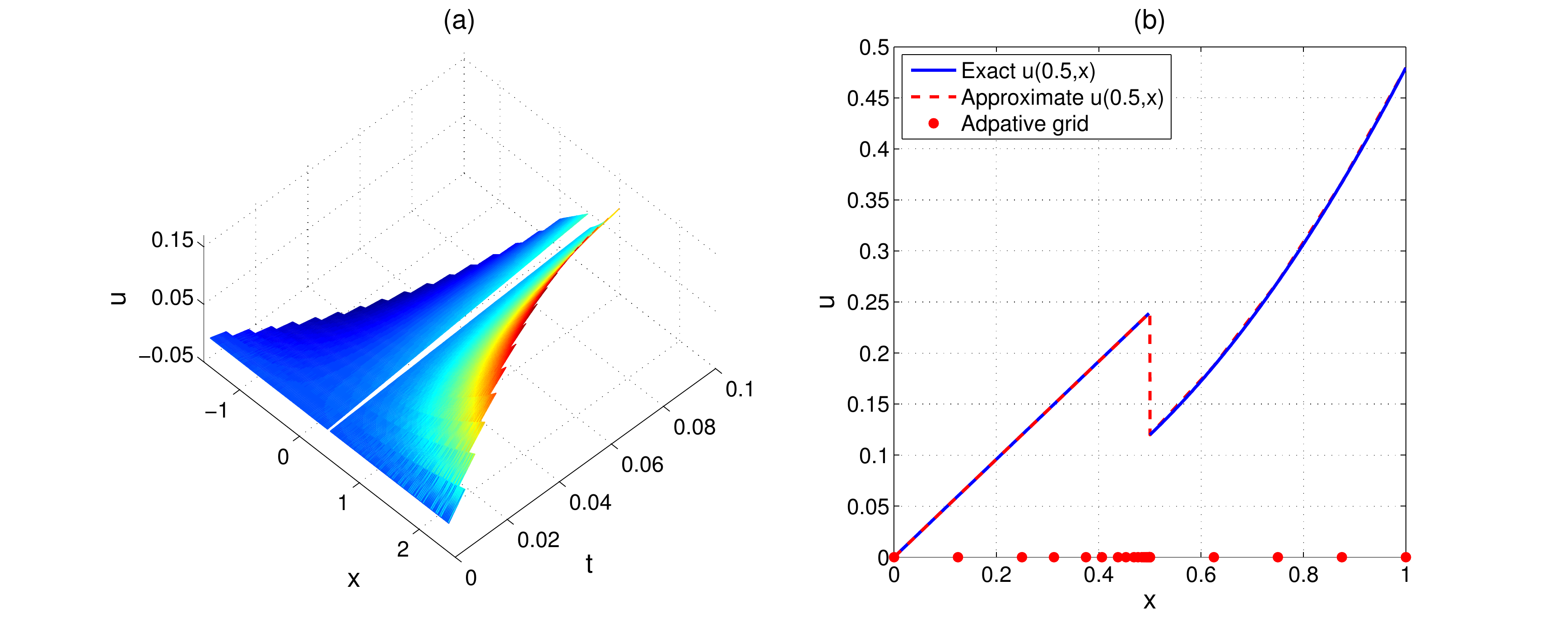}
\caption{(a) The surface of $u(t,x)$; (b) The exact solution $u(0.5,x)$ (solid line) and its approximation (dashed line) using 33 grid points (red dots).}\label{ex3:f1}
\end{center}
\end{figure}

\begin{figure}[h!]
\begin{center}
\includegraphics[scale =0.4]{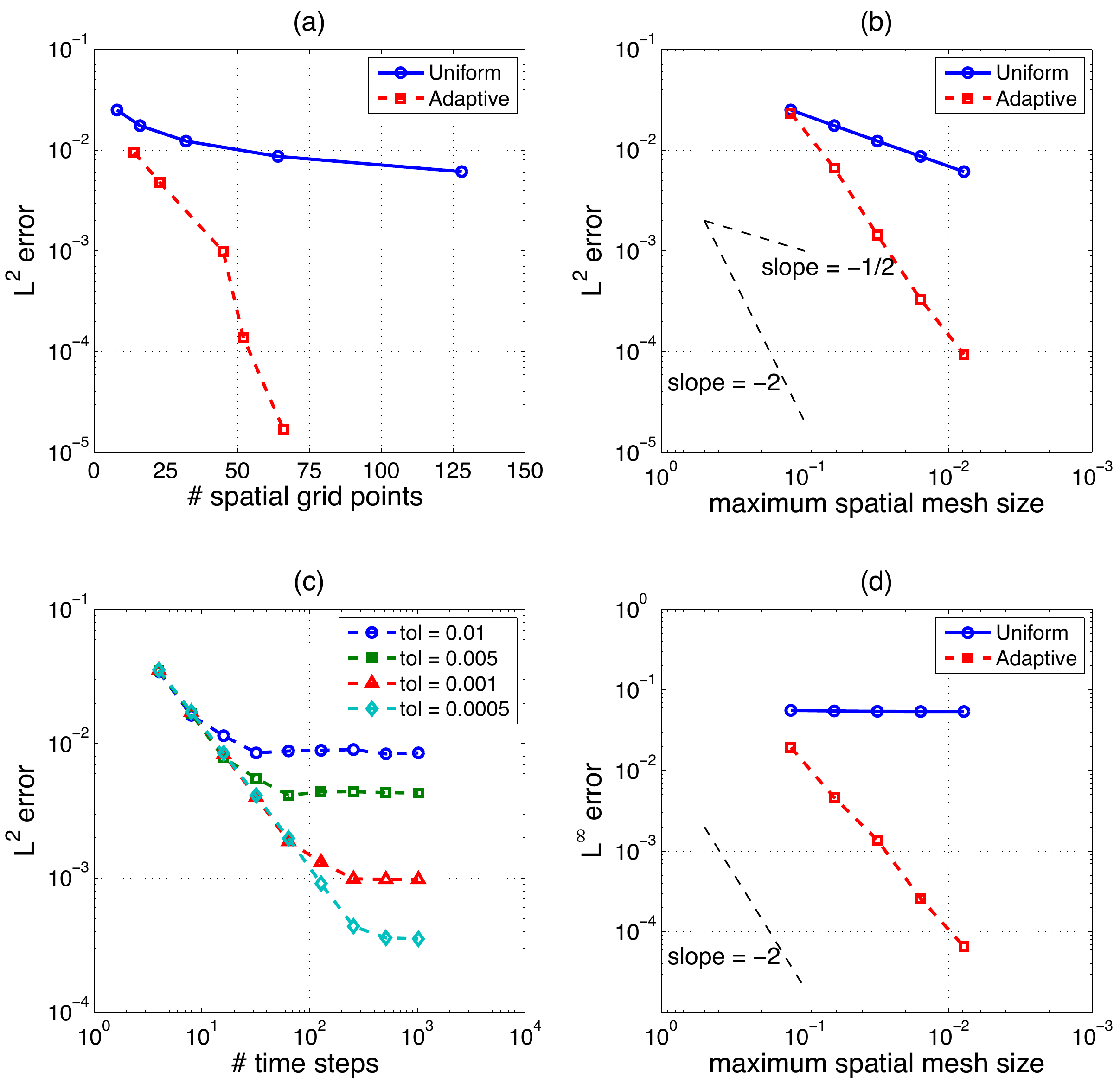}
\caption{Convergence of the $L^2$ norm of the error with respect to (a) the number of spatial grid points and (b) the spatial grid size $\Delta x$ using uniform and adaptive grids.
(c) Convergence with respect to $\Delta t$ for various tolerances of the adaptive spatial grid. (d) Convergence of the $L^\infty$ norm of the error (with the element containing the discontinuity ignored) with respect to $\Delta x$ using uniform and adaptive grids.}\label{ex3:f2}
\end{center}
\end{figure}

%
%
\section{Concluding remarks and future work}\label{sec:con}
In this work, we propose a novel stochastic numerical scheme for linear nonlocal diffusion problems based on the relationship between the PIDEs and a certain class of backward stochastic differential equations (BSDEs) with jumps. 
As discussed in \S \ref{sec:intro}, this effort focuses on solving low-dimensional nonlocal diffusion problems with high-order accuracy, where the potential applications will be non-Darcy flow in groundwater, or plasma transport in magnetically confined fields. However, our approach can be extended to solve moderately high-dimensional problems by incorporating our previous work on sparse-grid approximation \cite{Zhang:2013en,Gunzburger:2014hi}. 
Compared to standard finite element and collocation approaches for approximating linear nonlocal diffusion equations with integrable kernels, our method completely avoids the solution of dense linear systems.  Moreover, the ability to utilize 
high-order temporal and spatial discretization schemes, as well as efficient adaptive approximation, and the potential of massively parallel implementation, make our technique highly advantageous.
These assets have been verified by both theoretical analysis and numerical experiments. 

Our future efforts will focus on extending the proposed numerical schemes to the case of non-integrable kernels, e.g., fractional Laplacians, where the underlying L\`{e}vy processes cannot be described by compound Poisson processes. Moreover, the well-posedness of BSDEs on bounded domains, with volume constraints, is also critical for studying nonlocal diffusion equations on bounded domains.

\bibliographystyle{siam}

\begin{thebibliography}{10}

\bibitem{Anonymous:fk}
{\sc G.~Barles, R.~Buckdahn, and E.~Pardoux}, {\em {Backward stochastic
  differential equations and integral-partial differential equations}},
  Stochastics Stochastics Rep., 60 (1997), pp.~57--83.

\bibitem{Bouchard:2008cp}
{\sc B.~Bouchard and R.~Elie}, {\em {Discrete-time approximation of decoupled
  Forward{\textendash}Backward SDE with jumps}}, Stochastic Process. Appl., 118
  (2008), pp.~53--75.

\bibitem{Bouchard:2009il}
{\sc B.~Bouchard, R.~Elie, and N.~Touzi}, {\em {Discrete-time approximation of
  BSDEs and probabilistic schemes for fully nonlinear PDEs}}, in Advanced
  financial modelling, Walter de Gruyter, Berlin, 2009, pp.~91--124.

\bibitem{Bungartz:2004kx}
{\sc H.-J. Bungartz and M.~Griebel}, {\em {Sparse grids}}, Acta Numerica, 13
  (2004), pp.~1--123.

\bibitem{2011PhRvE..83a2105B}
{\sc N.~Burch and R.~B. Lehoucq}, {\em {Continuous-time random walks on bounded
  domains}}, Physical Review E, 83 (2011), p.~12105.

\bibitem{Chen:2011fd}
{\sc X.~Chen and M.~Gunzburger}, {\em {Continuous and discontinuous finite
  element methods for a peridynamics model of mechanics}}, Computer Methods in
  Applied Mechanics and Engineering, 200 (2011), pp.~1237--1250.

\bibitem{Delong:2013uk}
{\sc L.~Delong}, {\em {Backward Stochastic Differential Equations with Jumps
  and Their Actuarial and Financial Applications}}, BSDEs with Jumps, Springer
  Science {\&} Business, June 2013.

\bibitem{Du:2012hp}
{\sc Q.~Du, M.~Gunzburger, R.~B. Lehoucq, and K.~Zhou}, {\em {Analysis and
  approximation of nonlocal diffusion problems with volume constraints}}, SIAM
  Rev., 54 (2012), pp.~667--696.

\bibitem{Du:2013jn}
\leavevmode\vrule height 2pt depth -1.6pt width 23pt, {\em {A nonlocal vector
  calculus, nonlocal volume-constrained problems, and nonlocal balance laws}},
  Math. Models Methods Appl. Sci., 23 (2013), pp.~493--540.

\bibitem{Anonymous:HOINcelb}
{\sc Q.~Du, Z.~Huang, and R.~B. Lehoucq}, {\em {Nonlocal convection-diffusion
  volume-constrained problems and jump processes}}, Discrete and Continuous
  Dynamical Systems - Series B, 19 (2014), pp.~373--389.

\bibitem{Du:2013gf}
{\sc Q.~Du, L.~Ju, L.~Tian, and K.~Zhou}, {\em {A posteriori error analysis of
  finite element method for linear nonlocal diffusion and peridynamic models}},
  Math. Comp., 82 (2013), pp.~1889--1922.

\bibitem{Gunzburger:2014hi}
{\sc M.~D. Gunzburger, C.~G. Webster, and G.~Zhang}, {\em {Stochastic finite
  element methods for partial differential equations with random input data}},
  Acta Numerica, 23 (2014), pp.~521--650.

\bibitem{Hanson:2007ty}
{\sc F.~B. Hanson}, {\em {Applied stochastic processes and control for
  Jump-diffusions: modeling, analysis, and computation}}.
\newblock SIAM, Philadelphia, 2007.

\bibitem{Lejay:2007tu}
{\sc A.~Lejay, E.~Mordecki, and S.~Torres}, {\em {Numerical approximation of
  Backward Stochastic Differential Equations with Jumps}},  (2007).

\bibitem{Metzler:2000du}
{\sc R.~Metzler and J.~Klafter}, {\em {The random walk's guide to anomalous
  diffusion: a fractional dynamics approach}}, Physics Reports, 339 (2000),
  pp.~1--77.

\bibitem{Pardoux:1992jo}
{\sc E.~Pardoux and S.~Peng}, {\em {Backward stochastic differential equations
  and quasilinear parabolic partial differential equations}}, in Stochastic
  Partial Differential Equations and Their Applications, Springer Berlin
  Heidelberg, Berlin/Heidelberg, Jan. 1992, pp.~200--217.

\bibitem{Pardoux:1990ju}
{\sc E.~Pardoux and S.~G. Peng}, {\em {Adapted solution of a backward
  stochastic differential equation}}, Systems {\&} Control Letters, 14 (1990),
  pp.~55--61.

\bibitem{Platen:2010eo}
{\sc E.~Platen and N.~Bruti-Liberati}, {\em {Numerical Solution of Stochastic
  Differential Equations with Jumps in Finance}}, vol.~64 of Stochastic
  Modelling and Applied Probability, Springer Berlin Heidelberg, Berlin,
  Heidelberg, 2010.

\bibitem{Zhang:2013en}
{\sc G.~Zhang, M.~Gunzburger, and W.~Zhao}, {\em {A sparse-grid method for
  multi-dimensional backward stochastic differential equations}}, Journal of
  Computational Mathematics, 31 (2013), pp.~221--248.

\bibitem{Zhao:2006jx}
{\sc W.~Zhao, L.~Chen, and S.~Peng}, {\em {A New Kind of Accurate Numerical
  Method for Backward Stochastic Differential Equations}}, SIAM J. Sci.
  Comput., 28 (2006), pp.~1563--1581.

\bibitem{Zhao:2012ku}
{\sc W.~Zhao, Y.~Li, and G.~Zhang}, {\em {A generalized $\theta$-scheme for
  solving backward stochastic differential equations}}, Discrete Contin. Dyn.
  Syst. Ser. B, 17 (2012), pp.~1585--1603.

\bibitem{Zhao:2010ik}
{\sc W.~Zhao, G.~Zhang, and L.~Ju}, {\em {A stable multistep scheme for solving
  backward stochastic differential equations}}, SIAM J. Numer. Anal., 48
  (2010), pp.~1369--1394.

\end{thebibliography}

\end{document}